\documentclass[a4paper]{article}	
\usepackage[left=1.8cm,right=1.8cm,top=1.8cm,bottom=1.8cm]{geometry}
\usepackage[utf8]{inputenc}			
\usepackage{lmodern} 
\usepackage[numbers,square,sort&compress]{natbib}
\usepackage{graphicx}
\usepackage{amsmath,amsthm,amssymb,enumerate}
\usepackage{color}
\usepackage{dsfont}
\usepackage{multirow,booktabs}
\usepackage{array,longtable}
\usepackage{booktabs,subcaption,amsfonts,dcolumn}
\usepackage[labelformat=simple]{subcaption}

\allowdisplaybreaks

\catcode`\@=11 \@addtoreset{equation}{section}

\catcode`\@=12

\newtheorem{Theorem}{Theorem}[section]
\newtheorem{Definition}[Theorem]{Definition}

\theoremstyle{definition}
\newtheorem{Remark}[Theorem]{Remark}

\newcommand{\lbp}{\ell}

\newcommand{\xn}{x_n}
\newcommand{\xnn}{x_{n+1}}
\newcommand{\yn}{y_n}
\newcommand{\ynn}{y_{n+1}}
\newcommand{\intx}[3]{\int_{#1}^{#2} #3\, \mathrm{d}x}

\newcommand{\as}[1]{\renewcommand{\arraystretch}{#1}}

\renewcommand\arraystretch{1.15}

\begin{document}


\title{A new way of deriving implicit Runge-Kutta methods\\ based on repeated integrals}

\author{Hana Mizerov\' a\thanks{Corresponding author: hana.mizerova(at)fmph.uniba.sk. 
} 
		\and Katar\' ina Tvrd\' a $^{**}$
}
\date{}
\maketitle
\medskip

\centerline{$^*$ Department of Mathematical Analysis and Numerical Mathematics, Comenius University}
\centerline{Mlynsk\' a dolina, 842 48 Bratislava, Slovakia}

\bigskip
\centerline{$^{**}$ Faculty of Civil Engineering, Slovak University of Technology in Bratislava}
\centerline{Radlinsk\' eho 11, 810 05 Bratislava, Slovakia}

\abstract{Runge-Kutta methods have  an irreplaceable position among numerical methods designed to solve ordinary differential equations. Especially, implicit ones are suitable for approximating solutions of stiff initial value problems. We propose a new way of deriving  coefficients of implicit Runge-Kutta methods. This approach based on repeated integrals yields both new and well-known  Butcher's tableaux. We discuss the properties of newly derived methods and compare them with standard collocation implicit Runge-Kutta methods in a series of numerical experiments, including the Prothero-Robinson problem.


\medskip

\bigskip
\maketitle

\section{Introduction}\label{S-I}

Let us consider an initial value problem (IVP) of the form
\begin{subequations}\label{ivp}
\begin{align}
y'(x)&=f(x,y(x)), \  x\in [a,b],\label{ivp_eq}\\
y(a)&=y_a\label{ivp_ic}
\end{align} 
\end{subequations}
with $a<b,$ and $f:[a,b]\times \mathbb R \rightarrow \mathbb R,$  $y_a \in\mathbb R$ given.  The independent variable $x$ not always denotes time, but it is usually referred to as a time variable. Quantity $y$ represents the dependent variable.  Knowing the value $y_a$ of a quantity $y$ at time $x=a$ we want to predict what happens to $y$ as time evolves. Initial value problem is an essential tool in scientific computations and mathematical modelling. Therefore it is of a great interest to design efficient and reliable numerical schemes to approximate the exact solutions with a desired precision, especially in  the case of  stiff systems which are pertinent for many real-world problems described by differential equations.

Stiff problems form an important class of ordinary differential equations (ODEs). They appear in  many physical and chemical processes, for instance in fluid dynamics, chemical reaction kinetics, medicine, plasticity, neutron kinetics, porous media, gas transmission networks, transient magnetodynamics, in the study of spring and damping system,   and others, see \cite{nasa, gupta} and the references therein.
 It is difficult to define {\it stiffness}, but it is even more challenging to solve stiff problems numerically - standard numerical methods are unstable, requiring unexpectedly small discrete step in regions of smooth solution curve.

The famous Euler method published in 1768-1770 is based on a simple idea of a moving particle given by a differential equation: in a short period of time in which the velocity has not changed significantly, the change
in position will be approximately equal to the change in time multiplied by the initial velocity.
It was at the end of $19^{\rm th}$ and the beginning of $20^{\rm th}$ century when Runge, Heun and Kutta generalized the Euler method by allowing several evaluations of the derivative in one computational step. See, e.g, \cite{butcher_book} and the references therein. 
Nowadays Runge-Kutta methods (RK) represent an important  class of well-established numerical  methods for solving ODEs. The explicit ones with small stability regions turned out to be unsuitable for stiff problems - discrete step size has to meet stability rather than accuracy requirements \cite{BuFa}. However, {\it implicit  Runge-Kutta methods} with large stability regions are frequently used to solve stiff problems despite being more demanding from the computational point of view.

As already mentioned, stiff ODEs exhibit distressing behaviour when solved by classical numerical methods. Since explicit RK applied to stiff equations are usually unstable, and implementation of fully implicit RK is costly, there has been a growing interest in designing implicit methods with reduced computational costs. For instance, the so-called DIRK (diagonal implicit Runge-Kutta) and SDIRK (singly diagonal implicit Runge-Kutta) methods, see, e.g., \cite{nasa,FeSp} and the references therein. It is impossible to provide an exhausting literature overview on the design of new implicit Runge-Kutta(-type) methods. Let us mention, for instance, new optimized implicit-explicit Runge-Kutta methods \cite{imex}, new collocation methods based on weighted integrals  \cite{col_weight}, a two step fifth order RK for differential-algebraic equations \cite{skvortsov}, usage of interval analysis tools to compute coefficients of Runge-Kutta methods \cite{interval}, and  implicit seven stage tenth order Runge-Kutta methods based on Gauss-Kronrod-Lobatto quadrature formula \cite{7-10}.

In the present paper we provide a new way of deriving coefficients of implicit Runge-Kutta methods. Our approach yields both new and known schemes depending on the quadrature formula we choose.  In particular, we derive a stiffly accurate implicit four stage Runge-Kutta method with explicit first line of order 4. In addition, its experimental order of convergence in the case of linear IVPs is 6.   To the best of our knowledge there is no  result in the literature for computing  coefficients of Runge-Kutta methods based on repeated integration.

The paper is organized as follows:  we provide preliminary material in Section 2. Section 3 contains the main result - new strategy of designing implicit Runge-Kutta methods. Several numerical experiments are presented in Section 4. Conclusion is followed by Appendix containing the list of newly derived and some well-known implicit Runge-Kutta methods.

\section{Preliminaries}\label{S-M}

We provide a necessary mathematical apparatus, notation and known results used in the paper.

\subsection{A general Runge-Kutta method}\label{SS-RK}

Let $y(x)$ be the exact solution of \eqref{ivp}. Let $N\geq 1.$ We seek approximations $\yn$ of exact values $y(\xn)$ at points 
\begin{align*}
x_n=a+nh,\ h=\frac{b-a}{N},\ n=0,\ldots,N.
\end{align*}
Here $h$ denotes the {\em discrete step size}.
The  initial condition \eqref{ivp_ic} yields $y_0=y(x_0)=y(a)=y_a.$

A general Runge-Kutta method is a single step method  employing  $s$ stages $K_i$ in one step. To determine value $\ynn$ from known value $\yn$ we compute 
\begin{equation}\label{gIRK}
\begin{aligned}
\ynn &=\yn + h\sum_{i=1}^s b_iK_i \\
K_i&=f\left(\xn+h c_i, \yn +h\sum_{j=1}^s a_{ij}K_j\right), \ i=1,\ldots, s.
\end{aligned}
\end{equation}
The choice of coefficients $a_{i,j},$ $b_i$ and $c_i$  defines the method itself. 
Every Runge-Kutta method can be  represented by a practical {\it Butcher's tableau} \cite{butcher_history}, see Table~\ref{T-BT_gRK}.
\begin{table}[!ht]
\caption{Butcher's tableau of a general Runge-Kutta method}\label{T-BT_gRK}
\centering
$\begin{array}[b]{r|l}
\mathbf{c} & \mathbf{A}\\
\cline{1-2}
& \mathbf{b^T}
\end{array}
\quad 
\begin{array}{c}
=\\
\\
\end{array}
\quad
\begin{array}[b]{r|ccc}
c_1 & a_{1,1}& \ldots & a_{1,s}\\
. & . & \ldots & . \\
c_s & a_{s,1} & \ldots & a_{s,s} \\
\cline{1-4}
& b_1 & \ldots & b_s\\
\end{array}
$
\end{table}

Obviously, the structure of the matrix $\mathbf A$ decides whether the corresponding Runge-Kutta method is explicit (lower triangular matrix with zeros on the diagonal),  DIRK (lower triangular matrix), SDIRK (lower triangular matrix with the same diagonal elements) or implicit (full matrix).
Typical examples are the explicit  and implicit Euler  methods,  Heun's method - explicit trapezoidal rule, implicit trapezoidal rule, and the prominent explicit fourth-order Runge–Kutta method (RK4). See Table~\ref{T-BT_ex}.
\begin{table}[!ht]
\caption{Examples of well-known Runge-Kutta methods: explicit Euler, implicit Euler, Heun's method, implicit trapezoidal rule, RK4 (left to right)}\label{T-BT_ex}
\centering 
$\begin{array}[b]{c|c}
0 & 0\\
\hline
& 1
\end{array}$
\qquad
$\begin{array}[b]{c|c}
1 &1\\
\hline
& 1
\end{array}$
\qquad
$\begin{array}[b]{c|cc}
0&0&0\\
1&1& 0 \\
\hline 
& \frac 1 2 & \frac 1 2
\end{array}$
\qquad
$\begin{array}[b]{c|cc}
0&0&0\\
1&\frac 1 2 & \frac 1 2 \\
\hline 
& \frac 1 2 & \frac 1 2
\end{array}$
\qquad
$\begin{array}[b]{c|cccc}
0 & 0& 0& 0&0\\
\frac 1 2& \frac 1 2  &0 & 0 & 0 \\
\frac 1 2& 0& \frac 1 2 & 0& 0\\
1 & 0&0&1&0\\
\hline
& \frac 1 6 & \frac 1 3 & \frac 1 3 & \frac 1 6\\
\end{array}
$
\end{table}

\subsubsection{A-stability}

The concept of A-stability introduced by Dahlquist, originally used for multistep methods, suggests to   solve IVP 
\begin{equation}\label{IVP_A}
\begin{aligned}
y'(x)&=\lambda y(x)\\
y(0)&=1
\end{aligned}
\end{equation} 
with $\lambda \in \mathbb C,$ ${\rm Re}(\lambda)<0$ numerically. 
 Desirably our numerical solution has the same behaviour as the exact solution  $y(x)=\mathrm{e}^{\lambda x}$, namely
\begin{align}\label{stab_cond}
\yn \rightarrow 0 \ \mbox{ as } \ n \rightarrow \infty.
\end{align}
Applying a general RK on \eqref{IVP_A} we get, after one computational step, $\ynn=R(h\lambda)\yn,$ and by induction, $$\ynn=R(h\lambda)^n y_0.$$ Thus the stability condition \eqref{stab_cond} is equivalent to $|R(h\lambda)|<1.$

The function $R(z)$ is called the {\it stability function} and the set $S=\{z\in\mathbb C: |R(z)|\leq 1\}$ is called the {\it stability domain. }
If $S \supset \mathbb C^-$ (left half of complex plain), we say the method is {\it A-stable.}
The stability function of a general RK   reads 
\begin{align}
R(z)=\frac{{\rm det}(\mathbf I - z\mathbf A +z\mathbf e \mathbf b^T)}{{\rm det}(\mathbf I - z\mathbf A)},
\end{align}
where $\mathbf I,$ $\mathbf e$ are the identity matrix and the vector of ones, respectively, see \cite{butcher_book,HaWa}.

\subsubsection{Order conditions}

The {\it order} of a Runge-Kutta method is $p$ if and only if the local truncation error is $\mathcal{O}(h^{p+1}).$ There is no explicit Runge-Kutta method of order $p$ with $s=p$ stages for $p>4.$
The highest possible order of an implicit method with $s$ stages is $p=2s$ and is only attained by Gauss-Legendre methods. See \cite{butcher_bar} for more details on the so-called {\it Butcher's barriers}.

One of the great results of Butcher's theory for Runge-Kutta methods are the necessary and sufficient conditions to derive a new method of order $p$. These {\it order conditions} or {\it Butcher's rules} were firstly presented  and connected with the rooted tree theory by Butcher, see \cite{butcher_oc,butcher_oc2}. These conditions yield a system of equations with the unknowns $a_{i,j},$ $b_i$ and $c_i$. 
The number of constraints for each order increases exponentially and is different for one- and high-dimensional problems. Table~\ref{T-oc} contains 17 order conditions \cite{dor} to be satisfied by a Runge-Kutta method of order $p=5$. Orders $6,$ $7$ and $8$ impose 37, 85 and 200 constraints, respectively.
\begin{table}[!ht]
\centering
\caption{Order conditions for Runge-Kutta methods up to order $p=5$}\label{T-oc}
\as{1.0}
\begin{tabular}{llllll}
\toprule
 $p$ & \multicolumn{2}{l}{order conditions} & & \\
\midrule
1 &  $\displaystyle \sum_{i=1}^s b_i = 1$ & & \\
2 &  $\displaystyle \sum_{i=1}^s b_i c_i = \frac 1 2$& & \\
3 &  $\displaystyle \sum_{i=1}^s b_i c_i^2 = \frac 1 3,$& $\displaystyle  \sum_{i,j=1}^s b_i a_{i,j}c_j = \frac 1 6$ & \\
4 &  $\displaystyle \sum_{i=1}^s b_i c_i^3 = \frac 1 4,$& $\displaystyle  \sum_{i,j=1}^s b_ic_i a_{i,j}c_j = \frac 1 8,$& $\displaystyle  \sum_{i,j=1}^s b_i a_{i,j}c_j^2 = \frac{1}{12},$& $\displaystyle  \sum_{i,j,k=1}^s b_i a_{i,j}a_{j,k}c_k = \frac{1}{24}$\\
5 &  $\displaystyle \sum_{i=1}^s b_i c_i^4=\frac 1 5, $& $\displaystyle  \sum_{i,j=1}^s b_ic_i^2 a_{i,j}c_j = \frac{1}{10}, $& $\displaystyle \sum_{i,j=1}^s b_ic_i a_{i,j}c_j^2 = \frac{1}{15}, $& $\displaystyle  \sum_{i,j=1}^s b_i a_{ij}c_j^3 = \frac{1}{20}$ \\
&& $\displaystyle \sum_{i,j,k=1}^s b_i  a_{i,j}c_ja_{i,k}c_k = \frac{1}{20},$& $\displaystyle  \sum_{i,j,k=1}^s b_i c_i a_{i,j}a_{j,k}c_k = \frac{1}{30},$& $\displaystyle    \sum_{i,j,k=1}^s b_i a_{i,j} c_j a_{j,k}c_k = \frac{1}{40}, $ \\
& & &$\displaystyle  \sum_{i,j,k=1}^s b_i a_{ij}a_{j,k}c_k^2 = \frac{1}{60},$& $\displaystyle \sum_{i,j,k,m=1}^s b_i a_{i,j}a_{j,k}a_{k,m}c_m = \frac{1}{120}$ \\
\bottomrule
\end{tabular}
\end{table}
For linear ODEs the number of order conditions can be reduced, see \cite{ZiChi}. For instance, to obtain a sixth order method  it is enough to satisfy  16 conditions collected in Table~\ref{T-oc-lin}.
\begin{table}[!ht]
\centering
\caption{Order conditions for Runge-Kutta methods up to order $p=6$ for linear ODE's}\label{T-oc-lin}
\as{1.0}
\begin{tabular}{llllll}
\toprule
 $p$ & \multicolumn{2}{l}{order conditions} & & \\
\midrule
1 &  $\displaystyle \sum_{i=1}^s b_i = 1$ & & \\
2 &  $\displaystyle \sum_{i=1}^s b_i c_i = \frac 1 2$& & \\
3 &  $\displaystyle \sum_{i=1}^s b_i c_i^2 = \frac 1 3,$& $\displaystyle  \sum_{i,j=1}^s b_i a_{i,j}c_j = \frac 1 6$ & \\
4 &  $\displaystyle \sum_{i=1}^s b_i c_i^3 = \frac 1 4,$&  $\displaystyle  \sum_{i,j=1}^s b_i a_{i,j}c_j^2 = \frac{1}{12},$& $\displaystyle  \sum_{i,j,k=1}^s b_i a_{i,j}a_{j,k}c_k = \frac{1}{24}$\\
5 &  $\displaystyle \sum_{i=1}^s b_i c_i^4=\frac 1 5, $&$\displaystyle  \sum_{i,j=1}^s b_i a_{ij}c_j^3 = \frac{1}{20}$  &$\displaystyle  \sum_{i,j,k=1}^s b_i a_{ij}a_{j,k}c_k^2 = \frac{1}{60},$& $\displaystyle \sum_{i,j,k,m=1}^s b_i a_{i,j}a_{j,k}a_{k,m}c_m = \frac{1}{120}$ \\
6 & $\displaystyle \sum_{i=1}^s b_i c_i^5=\frac 1 6$  &$\displaystyle \sum_{i,j=1}^s b_i a_{ij}c_j^4 = \frac{1}{30}$ & $\displaystyle \sum_{i,j,k=1}^s b_i a_{ij}a_{j,k}c_k^3 = \frac{1}{120}$  &$  \displaystyle\sum_{i,j,k,m=1}^s b_i a_{i,j}a_{j,k}a_{k,m}c_m^2 = \frac{1}{360}$\\& & &  &$ \displaystyle\sum_{i,j,k,m,n=1}^s b_i a_{i,j}a_{j,k}a_{k,m}a_{m,n}c_n = \frac{1}{720}$\\
\bottomrule
\end{tabular}
\end{table}
A Runge-Kutta method is of order $p$ if and only if all order conditions up to order $p$ are satisfied.  

There are further constraints to be imposed in order to obtain a desired structure of the method, for instance:
\begin{itemize}
  \itemsep0em
\item[-] explicit: $a_{i,j}=0,$ $\forall\ j\geq 1$
\item[-] explicit first line: $a_{1,1}=\ldots=a_{1,s}=0$
\item[-] diagonal implicit: $a_{i,j}=0,$ $\forall j> 1$
\item[-] singly diagonal: $a_{1,1}=\ldots=a_{s,s}$
\item[-] stiffly accurate: $a_{s,i}=b_i,$ $\forall i=1, \ldots, s$
\item[-] fully  implicit: $a_{i,j}\neq0, $ $\forall i,j=1, \ldots, s.$
\end{itemize}  
Note that an additional consistency constraint (row-sum condition), $ c_i=\sum_{j=1}^s a_{i,j},$  is typically taken into account, but it is not necessary, see, e.g., \cite{Zl}.
We refer the interested reader to \cite{butcher_oc,butcher_oc2,HaNoWa} for more details.

\subsection{Standard implicit Runge-Kutta methods}

In what follows we give a brief note on the derivation of standard implicit Runge-Kutta methods (sIRK) which shall be compared to our newly derived methods in terms of order and accuracy. 

\subsubsection{Collocation methods}\label{SSS-sc}
Let us integrate the differential equation  \eqref{ivp_eq} over the interval $[\xn,x]$ to get 
\begin{align*}
y(x)-y(\xn)=\intx{\xn}{x}{f(x,y(x))}.
\end{align*}
We approximate the exact value $y(\xn)$ by $\yn$ and replace the integrand $f(x,y(x))$ by its unique Lagrange interpolation polynomial of degree at most $s-1$ corresponding to $s$ points $z_{n,i}=\xn+\tau_i h,$ $0\leq \tau_1<\ldots \tau_i < \ldots<\tau_{s}\leq 1,$ i.e.
\begin{align*}
y(x)\approx \yn +\intx{\xn}{x}{L_{s-1}(x)}=\yn+\sum_{i=1}^{s}f(z_{n,i},y(z_{n,i}))\intx{\xn}{x}{\lbp_i(x)},
\end{align*}
where
\begin{align*}
L_{s-1}(x)=\sum_{i=1}^{s} f(z_{n,i},y(z_{n,i}))\lbp_i(x), \quad \lbp_i(x)=\prod_{i\neq j=1}^{s} \left(\frac{x-z_{n,j}}{z_{n,i}-z_{n,j}}\right).
\end{align*}
The principle of collocation method requires the above identity to hold at any point $y(z_{n,i}).$ Thus, the approximations $y_{n,i}\approx y(z_{n,i})$ can be computed by a nonlinear system of equations
\begin{align*}
y_{n,i}=\yn+\sum_{j=1}^{s}f(z_{n,j},y(z_{n,j}))\intx{\xn}{z_{n,i}}{\lbp_j(x)}, \quad i=1,\ldots,s.
\end{align*}
In case $\tau_s=1$ we set $\ynn=y_{n,s},$ otherwise we set 
\begin{align*}
\ynn=\yn+\sum_{j=1}^{s}f(z_{n,j},y(z_{n,j}))\intx{\xn}{\xnn}{\lbp_j(x)}.
\end{align*}
Denoting $K_i=f(z_{n,i},y(z_{n,i})),$  $c_i=\tau_i,$ $i=1,\ldots, s,$ and
\begin{align*}
b_i&=\intx{\xn}{z_{n,s}}{ \lbp_i(x)},  \quad
a_{i,j}=\intx{\xn}{z_{n,i}}{\lbp_j(x)}   
\end{align*}
we get a standard implicit Runge-Kutta method  with $s$ stages. It shall be referred to as {\bf sIRK$s$}. For reader's convenience we list the corresponding Butcher's tableaux  for $s=2,3,4,5$ in Appendix, see Table~\ref{T-sIRK}.

\begin{Remark}\label{R-bi}
Choosing equally spaced points $z_{n,i}=\xn+h\frac{i-1}{s-1},$ $i=1,\ldots,s,$ $h=\xnn-\xn,$ means that coefficients $b_i$ are computed using a (closed) Newton-Cotes quadrature formula on the interval $[\xn,\xnn].$ 
\end{Remark}

\subsubsection{General implicit Runge-Kutta methods}\label{SSS-girk}

Every collocation method is a Runge-Kutta method but not every Runge-Kutta method is a collocation method. A general IRK can be derived using  the so-called {\it simplifying order conditions} introduced by Butcher \cite{butcher_soc}. In particular, the derivation of the coefficients $b_i$ and $a_{i,j}$ relies on three of them:
\begin{equation}\label{sor}
\begin{aligned}
B(p)&=\sum_{i=1}^s b_i c_i^{k-1}=\frac 1 k, \ \mbox{ for } \ k=1, \ldots, p\\
C(q)&=\sum_{j=1}^s a_{i,j} c_j^{k-1}=\frac{c_i^{k}}{k}, \ \mbox{ for } \ k=1, \ldots, q, \ i=1\ldots,s\\
D(r)&=\sum_{i=1}^s  b_i c_i^{k-1}a_{i,j}=\frac{b_j}{k}(1-{c_j^{k}}), \ \mbox{ for } \ k=1, \ldots, r, \ j=1\ldots,s.
\end{aligned}
\end{equation}
The coefficients $c_i$ are chosen according to the quadrature formula as specified below. The following theorem was proved in \cite{butcher_soc}.
\begin{Theorem}\label{theo-soc}
If a Runge-Kutta method satisfies $B(p),$ $C(q)$ and $D(r)$ with $p\leq q+r+1$ and $p\leq 2q+2,$ then its order of convergence is $p.$ 
\end{Theorem}

There are three groups of general IRK stemming from the choice of quadrature formula:
\begin{itemize}
\item[i)] {\it Gauss-Legendre methods} which are collocation methods based on the roots of the Legendre polynomial $P^*_s(x)$ shifted to $[0,1].$ The coefficients $b_i$ and $a_{i,j}$ are obtained from $B(s)$ and $C(s),$ respectively. It can be shown that they are of maximal order  $2s,$ where $s$ is the number of stages. 
\item[ii)] {\it Radau methods} are divided to two subgroups: Radau IA based on the Radau left quadrature formula: $c_1=0$ and remaining $c_i$ are the roots of $P_s^*(x)+P_{s-1}^*(x),$  and Radau IIA  based on the Radau right quadrature formula:  $c_s=1$ and remaining $c_i$ are the roots of $P_s^*(x)-P_{s-1}^*(x)$. Moreover, Radau IA impose conditions $B(s)$ and $D(s),$ while Radau IIA require conditions $B(s)$ and $C(s)$ to be satisfied. The well-known example of Radau IIA for $s=1$ is the {\it implicit Euler method}. The Radau methods are not collocation methods and are all of order  $2s-1.$
\item[iii)] {\it Lobatto methods,} also called {\it Lobatto III}, are based on the choice $c_1=0$, $c_s=1$ and remaining $c_i$ being the roots of the polynomial $P_s^*(x)-P_{s-2}^*(x).$ Nowadays there are three established subgroups of Lobatto methods based on the choice of the coefficients $a_{i,j}:$\newline Lobatto IIIA impose condition $C(s)$ and for $s=2$ yield the {\it implicit trapezoidal method}, Lobatto IIIB impose $D(s),$ and Lobatto IIIC impose conditions $a_{i,1}=b_i$ for $i=1,\ldots, s$ and $C(s-1).$  They are all of order $2s-2.$
\end{itemize}
We refer the reader to \cite{butcher_book,butcher_radau} for more details.  In Appendix we give Butcher's tableaux for selected Gauss-Legendre, Radau and Lobatto methods, see Table~\ref{T-RL}.

The simplifying order conditions  \eqref{sor} can be also used to define the {\em stage order}, cf. \cite{HaWa,rang,skvortsovA}. 
\begin{Definition}\label{def-so}
The stage order of a Runge-Kutta method is the maximal integer $\tilde{q}$ such that $B(p)$ and $C(q)$ hold for $p=1,\ldots,\tilde{q}$ and $q=1,\ldots,\tilde{q}.$
\end{Definition}

\subsection{Repeated integrals}\label{SS-ri}

We aim to derive coefficients for RK in order to approximate the  solution to IVP \eqref{ivp}.  Our new strategy  requires the moments of $y'(x)$ and $f(x,y(x))$ to be equal. It means that besides the standard integral identity typical for methods based on numerical integration, 
\begin{align}
\intx{\xn}{\xnn}{y'(x)} &= \intx{\xn}{\xnn}{f(x,y(x))} \label{first_int}
\end{align}
we want the following identities to be satisfied:
\begin{equation}
\begin{aligned}\label{moments}
\intx{\xn}{\xnn}{(\xnn-x)y'(x)} &= \intx{\xn}{\xnn}{(\xnn-x)f(x,y(x))} \\
\ldots&\ldots\ldots\\
\intx{\xn}{\xnn}{(\xnn-x)^{\gamma-1} y'(x)} &= \intx{\xn}{\xnn}{(\xnn-x)^{\gamma-1} f(x,y(x))} \\
\end{aligned}
\end{equation}
for arbitrary $\gamma\geq 1$ to be chosen later. As we shall see, it is useful to rewrite these moments equations as repeated integrals. This is possible thanks to 
Cauchy's formula for repeated integration \cite{cauchy} which allows us to compress $\gamma$ antiderivatives of one function to a single integral. See \cite{Fo} for the proof.
\begin{Theorem}[Cauchy's repeated integration formula]
Let $g$ be a continuous function on the real line. Then the $\gamma^{\rm th}$ repeated integral of $g,$ denoted by $g^{(-\gamma)},$ is given by a single integral,
\begin{align}
g^{(-\gamma)}:=\int_{a}^{b}{\int_{a}^{x_1}\ldots\int_a^{x_{\gamma-1}} g(x_\gamma)\, \mathrm{d}x_{\gamma}\ldots\mathrm{d}x_{1}} = \frac{1}{(\gamma-1)!}\int_a^b (b-t)^{\gamma-1}g(t) \,\mathrm{d}t. \label{CF}
\end{align}
\end{Theorem}
 From now on, to avoid confusion with the discretization points $x_n,$ $n=0,\ldots,N,$ we use $x$ for each integration variable on the left-hand side of \eqref{CF}.

By the Newton-Leibniz formula (the fundamental theorem of calculus) we have
\begin{align}\label{NL}
y(x)-y(\xn)=\intx{\xn}{x}{y'(x)}.
\end{align}
Due to \eqref{CF} identities \eqref{moments} for $\gamma\geq 2$ become 
\begin{align}\label{momentsC}
\intx{\xn}{\xnn}{\underbrace{\intx{\xn}{x}{\ldots\intx{\xn}{x}{y'(x)}}}_{(\gamma-1)-\mbox{times}}}&=\intx{\xn}{\xnn}{\underbrace{\intx{\xn}{x}{\ldots\intx{\xn}{x}{f(x,y(x))}}}_{(\gamma-1)-\mbox{times}}}.
\end{align}
In view of \eqref{NL} the integral on the left-hand side is nothing but 
\begin{align*}
\intx{\xn}{\xnn}{\underbrace{\intx{\xn}{x}{\ldots\intx{\xn}{x}{y'(x)}}}_{(\gamma-1)-\mbox{times}}}&=
\intx{\xn}{\xnn}{\underbrace{\intx{\xn}{x}{\ldots\intx{\xn}{x}{y(x)}}}_{(\gamma-2)-\mbox{times}}}-y(\xn)\intx{\xn}{\xnn}{\underbrace{\intx{\xn}{x}{\ldots\intx{\xn}{x}{1}}}_{(\gamma-2)-\mbox{times}}}\\
&=\intx{\xn}{\xnn}{\underbrace{\intx{\xn}{x}{\ldots\intx{\xn}{x}{y(x)}}}_{(\gamma-2)-\mbox{times}}}-y(\xn)\frac{h^{\gamma-1}}{(\gamma-1)!},\ h=\xnn-\xn.
\end{align*}  
Combining the above with \eqref{momentsC} we get 
\begin{equation}\label{nr}
\begin{aligned}
 h y(\xn)+\intx{\xn}{\xnn}{\intx{\xn}{x}{f(x,y(x))}}&=\intx{\xn}{\xnn}{y(x)}\\
 \frac{h^2}{2}y(\xn)+\intx{\xn}{\xnn}{\intx{\xn}{x}{\intx{\xn}{x}{f(x,y(x))}}}&=\intx{\xn}{\xnn}{\intx{\xn}{x}{y(x)}} \\
&\ldots \\
\frac{h^S}{S!}y(\xn)+\intx{\xn}{\xnn}{\underbrace{\intx{\xn}{x}{\ldots\intx{\xn}{x}{f(x,y(x))}}}_{S-\mbox{times}}}&=\intx{\xn}{\xnn}{\underbrace{\intx{\xn}{x}{\ldots\intx{\xn}{x}{y(x)}}}_{(S-1)-\mbox{times}}}
\end{aligned}
\end{equation}
 for $S=\gamma-1\geq 1.$ 

\subsubsection{Numerical quadrature for repeated integrals}

Another step towards deriving coefficients of new implicit RK is a suitable numerical approximation of repeated integrals in \eqref{nr}. We refer to  \cite{MT,NT}, where the detailed derivations of such quadrature formulas were presented. Here we briefly recall the results. 

\medskip
\noindent\textbf{Gauss quadrature for  repeated integral.}
Approximation of a repeated integral $g^{(-S)}$ by means of a general Gauss quadrature formula of order $2m-1$  reads
\begin{align}
g^{(-S)} \approx \frac{(b-a)^S}{S!}\sum_{j=1}^m w_{S,j} g(t_j),\label{mGQ}
\end{align} 
where  
\begin{align}
t_j=\frac{(b-a)(x_j+1)+2a}{2}, \qquad w_{S,j}=\frac{S w_j (1-x_j)^{S-1}}{2^S} \label{nwmGQ}
\end{align}
are the nodal Gaussian points shifted from $[-1,1]$ to $[a,b],$ and the weights depending on the number of repeated integrals $S$, respectively. 
 Here $x_j \in [-1,1]$ are the original nodal points and $w_j$ the corresponding weights of Gauss quadrature formula on $[-1,1].$ The reader might have noticed that \eqref{mGQ} with \eqref{nwmGQ} stems from Cauchy's formula \eqref{CF}. See \cite{MT} for more details.

\medskip
\noindent\textbf{Newton-Cotes quadrature for  repeated integral.} 
 Another possibility to approximate $g^{(-S)}$ is by the so-called modified Newton-Cotes formulas introduced in \cite{NT}. In this case, we firstly approximate function $g$ by its unique Lagrange interpolation  polynomial $L_{m-1}(x)$ corresponding to equidistant nodal points and then compute the weights directly, i.e. without applying Cauchy's formula \eqref{CF}.
 
  For the case of {\em closed Newton-Cotes formula} let $H=\frac{b-a}{m-1}$ and $ T_j=a+(j-1)H,\ j=1,\ldots,m.$ Then we write 
\begin{align*}
 g(x)\approx L_{m-1}(x)=\sum_{j=1}^m g(T_j)\ell_j(x), \quad \ell_j(x)=\prod_{j\neq i =1}^m \left(\frac{x-T_i}{T_j-T_i}\right),
\end{align*} 
where $\ell_j(x)$ are the basis Lagrange interpolation polynomials.
Consequently, for the repeated integral we obtain
\begin{subequations}\label{mNC}
 \begin{align}\label{mNCc}
g^{(-S)} \approx \sum_{j=1}^m g(T_j)\int_{a}^{b}\underbrace{{\int_{a}^{x}\ldots\int_a^{x}} \ell_j(x)\, \mathrm{d}x\ldots\mathrm{d}x }_{(S-1)-\mbox{times}}= H^S \sum_{j=1}^m W_{S,j} g(T_j),\quad H^S W_{S,j}=\ell_j^{(-S)}.
\end{align} 
An analogous formula holds for {\em open Newton-Cotes formulas} which do not include the endpoints $a$ and $b.$ Indeed, 
 \begin{align}\label{mNCo}
g^{(-S)} \approx H^S \sum_{j=1}^{m-2} W_{S,j} g(T_j),\quad H^S W_{S,j}=\ell_j^{(-S)}, \quad T_j=a+jH,\ j=1,\ldots,m-2
\end{align} 
\end{subequations}
with the interval $[a,b]$ being divided into $m-1$ subintervals of equal length.

\section{New families of implicit Runge-Kutta methods}\label{S-new}

Our motivation is to solve  initial value problem \eqref{ivp} by an implicit Runge-Kutta method (IRK). We propose a new way of deriving the entries of $\mathbf{A},$ $\mathbf{b}$ and $\mathbf{c},$ cf. Table~\ref{T-BT_gRK}. Instead of focusing on simplifying order conditions \eqref{sor} we employ the moment approach and repeated integrals as described in Subsection~\ref{SS-ri} to determine the coefficients $a_{i,j},$ $b_i$ and $c_i.$ 

\medskip
\noindent
We proceed in four steps:
\begin{itemize}
\itemsep0em
\setlength{\itemindent}{.2in}
\item[]{\bf(Step 1)}  choose a numerical quadrature  and determine $c_i$ accordingly
\item[]{\bf(Step 2)}  compute $b_i$ by the chosen quadrature formula 
\item[]{\bf(Step 3)}  for the number of unknown stages write the corresponding number of identities \eqref{nr}  and apply chosen quadrature formula on repeated integrals
\item[]{\bf(Step 4)}  solve the resulting  system of algebraic equations for unknowns $y_{n,i}$ and get $a_{i,j}$  
\end{itemize} 
We shall follow these steps and detail the procedure for three different cases: the cases of closed and open Newton-Cotes  quadrature formulas (NC) as well as the case of Gaussian quadratures of Legendre-, Radau- and  Lobatto-type. 

\subsection{New way based on closed Newton-Cotes formulas}
\label{SS-cNC}
\noindent
{\bf(Step 1)} For a closed NC with the nodal pints $z_{n,i}=\xn+h\frac{i-1}{s-1},$ $i=1,\ldots,s,$ we have $$c_i=\frac{i-1}{s-1}, \ i=1,\ldots,s.$$

\noindent
{\bf(Step 2)} Integrating \eqref{ivp_eq} over $[\xn,\xnn]$ and applying the closed NC from Step 1 to approximate  the integral of $f(x,y(x))$ over $[\xn,\xnn]$  we get 
\begin{align}\label{s2}
\ynn=\yn+h\sum_{i=1}^s b_i K_i, \quad b_i=\intx{\xn}{\xnn}{ \ell_i(x)}, \ i=1,\ldots,s.
\end{align}
Recall that  $K_i=f(z_{n,i},y(z_{n,i})).$ 
 The coefficients $b_i$ are simply the integrals of the corresponding Lagrange basis interpolation polynomials, analogously as in the case of standard collocation IRK, cf. Remark~\ref{R-bi}. Indeed, replacing $f(x,y(x))$ by its unique Lagrange interpolation polynomial corresponding to the nodal points from  Step 1 we get
\begin{align*}
y(\xnn)=y(\xn)+ \intx{\xn}{\xnn}{f(x,y(x))} \approx \ynn=\yn + \sum_{i=1}^s f(z_{n,i},y(z_{n,i})) \intx{\xn}{\xnn}{ \ell_i(x)}.
\end{align*}
The difference compared to collocation methods from Subsubsection~\ref{SSS-sc} lies in deriving the entries of the matrix $\mathbf{A}.$ 

\noindent
{\bf(Step 3)} Let $s\geq 3.$ Since $z_{n,1}=\xn$ and $z_{n,s}=\xnn,$ we  have $S=(s-2)$ unknown stages. Thus we consider the first $S$ equations of \eqref{nr} to which we apply the modified Newton-Cotes quadrature \eqref{mNCc} with $H=\frac{h}{s-1}.$ 
This results in a system of $S$ equations with the unknowns $y(z_{n,2}),$ $\ldots,$ $y(z_{n,s-1}),$ 
\begin{equation}\label{nr_plug}
\begin{aligned}
  h\yn+ H^2\sum_{i=1}^s W_{2,i}K_i&=H\sum_{i=1}^s W_{1,i} y(z_{n,i})\\
 \frac{h^2}{2}\yn+H^3\sum_{i=1}^s W_{3,i}K_i&=H^2\sum_{i=1}^s W_{2,i}y(z_{n,i})\\
&\ldots \\
\frac{h^S}{S!}\yn+H^{S+1}\sum_{i=1}^s W_{S+1,i}K_i&=H^S\sum_{i=1}^s W_{S,i}y(z_{n,i}),
\end{aligned}
\end{equation}
where $K_i=f(z_{n,i},y(z_{n,i}))$ are parameters.
 Note  that $y(z_{n,1})=y(\xn)\approx\yn$ is known from the previous computational step and $y(z_{n,s})=y(\xnn)\approx\ynn$ is given by \eqref{s2} in Step 2.

\noindent
{\bf(Step 4)} Each $y(z_{n,i}),$ $i=2,\ldots,s-1$ as a solution to  \eqref{nr_plug} taking into account \eqref{s2} can be expressed in the form
\begin{align*}
y(z_{n,i})=\yn+h\sum_{j=1}^s a_{i,j}K_j, \ i=2, \ldots, s-1,
\end{align*}
from which we get the coefficients $a_{i,j},$ $i=2,\ldots,s-1,$ $j=1,\ldots,s.$
Due to $y(z_{n,1})=y(\xn)\approx\yn$ we have $a_{1,j}=0,$ $j=1,\ldots,s,$ and from \eqref{s2} simply $a_{s,j}=b_j,$ $j=1,\ldots,s.$  

\medskip

For $\mathbf{s=2}$ we obviously get the Lobatto IIIA with 2 stages and order 2, the so-called {\it implicit trapezoidal method}. Since $K_1=f(\xn,y(\xn))$ and $K_2=f(\xnn,y(\xnn)),$ there are no unknown stages and thus no need to use repeated integrals \eqref{nr}. 

For $\mathbf{s=3}$ there is only one unknown stage $K_2=f(z_{n,2},y(z_{n,2}))$ with $z_{n,2}=\xn+\frac{1}{2}h.$ Thus, from  the first equation of \eqref{nr_plug} we get 
\begin{align*}
h \yn+ \frac{2H^2}{3}\left(K_1+2K_2\right)=\frac{H}{3}\left(\yn+2y(z_{n,2})+\ynn\right), \ H=\frac{h}{2}.
\end{align*}
In addition, \eqref{s2} results in
\begin{align*}
\ynn =\yn + \frac{h}{6}\left(K_1+4K_2+K_3 \right).
\end{align*}
Therefore, 
\begin{align*}
y(z_{n,2}) &= \left(\frac{3}{2}-\frac{1}{8} -\frac{1}{4}\right)\yn+ \frac{h}{4}\left(K_1+2K_2\right)- \frac{h}{24}\left(K_1+4K_2+K_3 \right) \\
& = \yn+\frac{h}{24}\left(5K_1+8K_2-K_3 \right).
\end{align*}
We conclude
\begin{align*}
 \mathbf{A}=\left(\begin{array}{ccc}
0 & 0 & 0\\
\frac{5}{24}&\frac{8}{42}&-\frac{1}{24}\\
\frac{1}{6}&\frac{2}{3}&\frac{1}{6}\\
\end{array}\right),\quad \mathbf{b}^T=\left(\frac{1}{6},\frac{2}{3},\frac{1}{6}\right),
 \quad \mathbf{c}=\left(
 \begin{array}{c}
0 \\
\frac{1}{2}\\
1\\
\end{array}\right).
\end{align*}
This method is again known as Lobatto IIIA.

For $\mathbf{s=4}$ we deal with the well-known ${ 3/ 8}$-Simpson rule.
It means we divide the interval $[\xn,\xnn]$ into three subintervals of equal length $H=\frac{h}{3}.$ Then 
\begin{align*}
\intx{\xn}{\xnn}{f(x,y(x))} \approx \frac 3 8 H \left(K_1+3K_2+3K_3+ K_4\right),
\end{align*} 
where $z_{n,i}=x_n+(i-1)H,$ $i=1,2,3,4.$ Thus $c_i=\frac{i-1}{3},$ $i=1,2,3,4,$ and
Step 2 yields 
\begin{align}\label{g1}
\ynn=\yn +  \frac 3 8 H \left(K_1+3K_2+3K_3+ K_4\right).
\end{align}
It means $b_1=\frac{1}{8}=b_4$ and $b_2=\frac{3}{8}=b_3.$ 
Since the stages $K_2$ and $K_3$ including the values $y(z_{n,2})$ and $y(z_{n,3})$ are at this point unknown, we need two additional equations to close the system. Therefore we consider the first two identities of \eqref{nr_plug} to get
\begin{align*}
 3H\yn&+\frac{3}{40} H^2 \left(13K_1+36K_2+9K_3+ 2K_4\right)=\frac 3 8 H\left(\yn+3y(z_{n,2})+3y(z_{n,3})+\ynn\right)\\
 \frac{9H^2}{2}\yn &+\frac{9}{80} H^3 \left(12K_1+27K_2+0 K_3+ K_4 \right)=\frac{3}{40} H^2 \left(13\yn+36y(z_{n,2})+9y(z_{n,3})+ 2\ynn\right).
\end{align*}
Let us denote $\alpha_j\in\{1,3,3,1\},$ $\beta_j\in\{13,36,9,2\},$ $\gamma_j\in\{12,27,0,1\}.$
 Using \eqref{g1} we further rewrite the latter equations to get 
\begin{subequations}\label{nnr}
\begin{align}
 y(z_{n,2})&= 2\yn + \frac{2}{3} H \sum_{i=1}^4\left(\frac{\beta_i}{10}-\frac{3\alpha_i}{16}\right)K_i-y(z_{n,3})\label{nnr1}\\
 y(z_{n,3})&= 5\yn + \frac{1}{6} H \sum_{i=1}^4\left(\gamma_i-\frac{\alpha_i}{2}\right)K_i-y(z_{n,2}).\label{nnr2}
\end{align}
\end{subequations}
Now we substitute \eqref{nnr2} into \eqref{nnr1} to obtain
\begin{align}\label{g2}
y(z_{n,2})&= \yn + \frac{H}{18}  \sum_{i=1}^4\left(\frac{\alpha_i}{4}-\frac{2\beta_i}{5}+\gamma_i\right)K_i, 
\end{align}
and analogously, \eqref{nnr1} into \eqref{nnr2} to get
\begin{align}\label{g3}
y(z_{n,3})&= \yn + \frac{H}{18}\sum_{i=1}^4\left(-\frac{5\alpha_i}{2}+\frac{8\beta_i}{5}-\gamma_i\right)K_i. 
\end{align}
Now, let $$A_i=\frac{\alpha_i}{4}-\frac{2\beta_i}{5}+\gamma_i\in \left\{\frac{141}{20}, \frac{267}{20}, -\frac{57}{20}, \frac{9}{20}\right\}, \quad  B_i=-\frac{5\alpha_i}{2}+\frac{8\beta_i}{5}-\gamma_i\in\left\{\frac{63}{10}, \frac{231} {10}, \frac{69}{10}, -\frac{3}{10}\right\}.$$
Hence \eqref{g1}, \eqref{g2} and \eqref{g3} become 
\begin{align*}
\ynn&=\yn+\frac{3H}{8} \sum_{i=1}^4\alpha_iK_i, \quad 
y(z_{n,2})= \yn + \frac{H}{18}  \sum_{i=1}^4A_iK_i, \quad 
y(z_{n,3})= \yn + \frac{H}{18}\sum_{i=1}^4B_iK_i. 
\end{align*}
Recall $H=\frac h 3$ and
$K_i=f(z_{n,i},y(z_{n,i})).$  We have derived four stage IRK given by
\begin{align*}
\ynn&=\yn+ h \sum_{i=1}^4\frac{\alpha_i}{8}K_i\\
K_1&=f(\xn,\yn)\\
K_2&=f\left(\xn+\frac h 3,\yn + h  \sum_{i=1}^4 \frac{A_i}{54}K_i\right) \\
  K_3&=f\left(\xn+\frac{2}{3}h,\yn + h\sum_{i=1}^4\frac{B_i}{54}K_i\right)\\
   K_4&=f\left(\xn+h, \yn + h \sum_{i=1}^4\frac{\alpha_i}{8}K_i\right) .
\end{align*}
We shall refer to it as {\bf nIRK4}. Note that Lobatto IIIA with 4 stages differs in the choice of $c_i$ and sIRK4 differs in the coefficients $a_{2,j},$ $a_{3,j},$ $j=1,2,3,4.$ 

\begin{Theorem}[properties of  nIRK4]\label{theo-4}
Newly derived four stage implicit Runge-Kutta method  given by Butcher's tableau 
\begin{align}\label{BT_nIRK4}
\begin{array}{c|cccc}
0 & 0 & 0 & 0 & 0 \\
\frac 1 3 & \frac{47}{360} & \frac{89}{360} & -\frac{19}{360} & \frac{3}{360} \\
\frac 2 3 & \frac{21}{180} & \frac{77}{180} & \frac{23}{180} & -\frac{1}{180} \\
1 & \frac 1 8 & \frac 3 8 & \frac 3 8 & \frac 1 8 \\
\hline 
 & \frac 1 8 & \frac 3 8 & \frac 3 8 & \frac 1 8\\
\end{array}
\end{align}
 is  stiffly accurate with explicit first line. It is  A-stable with the order of convergence $p=4,$  and the stage order $\tilde{q}=3.$
\end{Theorem}
\begin{proof}
Obviously, the first line is explicit and since $a_{s,j}=b_j,$ $j=1,\ldots,s$ the method is stiffly accurate according to \cite{butcher_oc}. 
The method is indeed A-stable. The stability function  reads
\begin{align*}
R(z)= -\frac{z^3+ 12z^2+ 60 z +120}{z^3- 12z^2+ 60 z-120},
\end{align*}
and direct calculation for $z=u+iv \in \mathbb C,$  yields
\begin{align*}
|R(z)| \leq 1 \Leftrightarrow   u(u^4+v^4+2u^2v^2+70u^2+30v^2+600) \leq 0
\end{align*}
which is only true for $u\leq 0.$ Thus the stability domain is  $S=\{z\in \mathbb C: {\rm Re}(z)\leq 0\}.$
It is easy to see that the method satisfies the simplifying order conditions $B(p),$ $C(q)$ and $D(r)$ with $p=4,$ $q=3$ and $r=0,$ which by Theorem~\ref{theo-soc} and Definition~\ref{def-so}  means the order is $p=4,$ and the stage order is $\tilde{q}=q=3.$
Another possibility is to check that all 8 order conditions from Table~\ref{T-oc} up to order $p=4$ are satisfied. 
\end{proof}

For $\mathbf{s\geq 5}$ one can analogously as above derive IRK  based on closed Newton-Cotes formulas.
 We give  Butcher's tableau and the corresponding order of {\bf nIRK5} in Table~\ref{T-NCclosed}.  
 
 Note that {\it all {\bf nIRK$s$} are stiffly accurate} and have {\em explicit first line} which follows from the derivation described above.

\subsection{New way based on open Newton-Cotes formulas}
\label{SS-oNC}
The only difference in deriving the coefficients of IRK by an open Newton-Cotes formula compared to the closed one is that the end points of the interval $[\xn,\xnn]$ are not included. It means the resulting method neither has the explicit first line nor is stiffly accurate according to \cite{butcher_oc}.

\smallskip
\noindent
{\bf (Step 1)} In order to derive an $s$ stage Runge-Kutta method, we need to consider $s+1$ subintervals of $[\xn,\xnn]$ with the corresponding nodal points $z_{n,i}=\xn+h\frac{i}{s+1},$ $i=1, \ldots, s.$ Obviously, $$c_i=\frac{i}{s+1},\ i=1,\ldots, s.$$

\noindent
{\bf (Step 2)} By the same reason as before, it holds that 
\begin{align}\label{s2o}
\ynn=\yn+h\sum_{i=1}^s b_i K_i, \quad b_i=\intx{\xn}{\xnn}{ \ell_i(x)}, \ i=1,\ldots,s.
\end{align}
Note the difference compared to \eqref{s2} lies in $\ell_i(x)$ being polynomials defined for different nodal points and   of different degree.

\noindent
{\bf (Step 3)}  For $s\geq 2$ we have $S=s$ unknown stages and the corresponding first $s$ equations of \eqref{nr} to be solved. Note that $y(z_{n,1})\neq y(\xn)$ and $y(z_{n,s})\neq y(\xnn).$ We approximate the integrals in \eqref{nr} by the modified open Newton-Cotes quadrature formula \eqref{mNCo} which results in 
\begin{equation}\label{nr_plug_2}
\begin{aligned}
  h\yn+ H^2\sum_{i=1}^s W_{2,i}K_i&=H\sum_{i=1}^s W_{1,i} y(z_{n,i})\\
 \frac{h^2}{2}\yn+H^3\sum_{i=1}^s W_{3,i}K_i&=H^2\sum_{i=1}^s W_{2,i}y(z_{n,i})\\
&\ldots \\
\frac{h^S}{S!}\yn+H^{S+1}\sum_{i=1}^s W_{S+1,i}K_i&=H^S\sum_{i=1}^s W_{S,i}y(z_{n,i}),
\end{aligned}
\end{equation}
where $K_i=f(z_{n,i},y(z_{n,i}))$ are parameters, and $y(z_{n,1}),\ldots,y(z_{n,s})$ are the unknowns.
Note that $y(\xn)\approx\yn$ is known from the previous computational step and $H=\frac{h}{s+1}.$ 

\noindent
{\bf (Step 4)} We solve system \eqref{nr_plug_2} and as a result we obtain the coefficients $a_{i,j},$ $i,j=1,\ldots,s$ from 
  \begin{align*}
y(z_{n,i})=\yn+h\sum_{j=1}^s a_{i,j}K_j, \ i=1, \ldots, s.
\end{align*}

For $s=3,4$ the newly derived IRK are listed in Table~\ref{T-NCopen} and are referred to as {\bf nIRK3o} and {\bf nIRK4o}. They are both A-stable and of order $p=4$.  

As opposed to the methods based on closed Newton-Cotes formulas, {\bf nIRK$s$o} are neither stiffly accurate nor have explicit first line.

\begin{Remark}
Dividing the interval $[\xn,\xnn]$ into $s$ subintervals of equal length $H=\frac{h}{s}$ yields either $(s+1)$ or $(s-1)$ stage method for closed or open NC, respectively. 
\end{Remark} 

\subsubsection*{New way based on Newton-Cotes formulas combined with Cauchy's formula}

As we have seen, weights in modified Newton-Cotes quadrature formulas \eqref{mNC} are obtained as repeated integrals of the basis Lagrange interpolation polynomials. However, we can also rewrite them using Cauchy's repeated  integration formula \eqref{CF} and then compute the weights.  
When we use this approach in \eqref{nr} in {Step 3} we derive new yet different implicit Runge-Kutta methods. Their Butcher's tableaux and corresponding orders for $s=2,3,4,5$ in the case of closed formulas  are given in Tables~\ref{T-NCclosed}, \ref{T-RL} and for $s=3,4$ in the case of open formulas are given in Table~\ref{T-NCopen}. They are refereed to as {\bf nIRK$s$c} and {\bf  nIRK$s$oc}, respectively. Here ``{\bf c}'' stands for ``Cauchy''.

\subsection{New way based on Gauss quadratures}
\label{SS-ngq}

It turns out that our approach based on repeated integrals combined with Gauss quadratures yields the well-known implicit Runge-Kutta methods. For completeness, let us briefly summarize the results. 

\subsubsection*{Gauss-Legendre quadrature}
 
 \noindent
{\bf (Step 1)}
 The choice of $c_i$ is obvious:  the roots of the Legendre polynomial $P_s^*(x)$ shifted to $[0,1].$ 

 \noindent
{\bf (Step 2) }
 The coefficients $b_i$ are exactly the weights $w_i$ of the corresponding Gauss-Legendre quadrature. Indeed,  
\begin{align*}
\intx{\xn}{\xnn}{f(x,y(x))} \approx h\sum_{i=1}^s f(z_{n,i},y(z_{n,i})) w_i =h\sum_{i=1}^s b_i K_i.
\end{align*}
Recall $K_i=f(z_{n,i},y(z_{n,i}))$ with $z_{n,i}=\xn+c_i h,$ $i=1,\ldots,s.$ 

 \noindent
{\bf (Step 3)}
Let $s\geq 2.$ We realize that $z_{n,1}\neq \xn$ and $z_{n,s}\neq \xnn.$ Thus, we have $S=s$ unknown stages to be determined in order to compute $\ynn$ given $\yn.$ To this end, we consider the first $s$ identities of \eqref{nr}. The repeated integrals of $f(x,y(x))$ and $y(x)$ are now approximated by means of \eqref{mGQ}. The system reads
\begin{equation}\label{nr_plug_g}
\begin{aligned}
  h\yn+ \frac{h^2}{2!}\sum_{i=1}^s w_{2,i}K_i&=h\sum_{i=1}^s w_{1,i} y(z_{n,i})\\
 \frac{h^2}{2}\yn+\frac{h^3}{3!}\sum_{i=1}^s w_{3,i}K_i&=\frac{h^2}{2!}\sum_{i=1}^s w_{2,i}y(z_{n,i})\\
&\ldots \\
\frac{h^s}{s!}\yn+\frac{h^{s+1}}{(s+1)!}\sum_{i=1}^s w_{s+1,i}K_i&=\frac{h^s}{s!}\sum_{i=1}^s w_{s,i}y(z_{n,i}).
\end{aligned}
\end{equation}

 \noindent
{\bf (Step 4)} Solving the underlying system yields $y(z_{n,1}),$ $\ldots,$ $y(z_{n,s}),$ in the form 
  \begin{align*}
y(z_{n,i})=\yn+h\sum_{j=1}^s a_{i,j}K_j, \ i=1, \ldots, s,
\end{align*}
 and finally we have the coefficients $a_{i,j},$ $i,j=1,\ldots,s.$

\subsubsection*{Gauss-Radau quadrature}

\noindent
{\bf (Step 1)}
 Set $c_1=0$ and $c_i$ stemming  from left Radau quadratures or set $c_s=1$ and $c_i$ stemming  from right Radau quadratures as described above in Subsubsection~\ref{SSS-girk}.

\noindent 
{\bf (Step 2)} Analogously as for the Gauss-Legendre quadrature, the coefficients $b_i$ coincide with the weights of Radau left quadrature formula or Radau right quadrature formula, respectively.

 \noindent 
{\bf (Step 3)} In both cases we have $S=s-1$ unknown stages. We consider the first $s-1$ equations of \eqref{nr} where the repeated integrals are approximated by \eqref{mGQ} using the corresponding nodal points and weights of left or right Radau quadratures, respectively. 

\noindent 
{\bf (Step 4)}
Solving the system of $s-1$ first equations of \eqref{nr_plug_g} with unknowns $y(z_{n,2}),$ $\ldots,$ $y(z_{n,s})$ in the case of left Radau quadrature, or with unknowns $y(z_{n,1}),$ $\ldots,$ $y(z_{n,s-1})$ in the case of right Radau quadrature, yields the coefficients $a_{i,j},$ $i,j=1,\ldots,s.$ Note that  $a_{1,j}=0$ for left Radau quadrature
and $a_{s,j}=b_j$  for right  Radau quadrature formulas.

\subsubsection*{Gauss-Lobatto quadrature}
 
 \noindent 
{\bf (Step 1)}
We set $c_1=0,$  $c_s=1$ and $c_i,$ $i=2,\ldots,s-2,$ to be the roots  of  $P_s^*(x)-P_{s-2}^*(x).$ 

 \noindent 
{\bf (Step 2)}
The coefficients $b_i$ coincide with the weights of chosen Lobatto quadrature formula, and again we have
\begin{align}\label{s2L}
\ynn=\yn+h\sum_{i=1}^s b_i K_i.
\end{align}

 \noindent 
{\bf (Step 3)}
We need to consider $S=s-2$ equations of \eqref{nr} together with \eqref{s2L} on noting that $\yn$ is known from the previous computational step. We approximate the repeated integrals by means of \eqref{mGQ} using the nodal points and weights corresponding to chosen Gauss-Lobatto quadrature formula. 

 \noindent 
{\bf (Step 4)}
Coefficients $a_{i,j},$ $i,j=2,\ldots,s-2,$ are obtained from 
 \begin{align*}
y(z_{n,i})=\yn+h\sum_{j=1}^s a_{i,j}K_j, \ i=2, \ldots, s-2
\end{align*}
which solve the system from Step 3  taking into account \eqref{s2L}.
 Since $z_{n,1}=\xn$ we have the explicit first line,  $a_{1,j}=0,$ $j=1,\ldots,s.$ Finally, $a_{s,j}=b_j,$ $j=1,\ldots,s$ because $z_{n,s}=\xnn.$ 

The resulting IRK of Gauss-type ({\bf nIRK-G$s$},  {\bf nIRK-RI$s$}, {\bf nIRK-RII$s$}, {\bf nIRK-L$s$} according to type of Gauss quadrature) are listed in Table~\ref{T-RL}  and are classified according to existing families of general Gauss-type IRK.

Note that the {\it advantage of our proposed approach} compared to the derivation of general IRK based on Gauss quadratures as described in Subsubsection~\ref{SSS-girk} lies in the computation of $a_{i,j}$ and $b_i.$ For instance, in the case of Gauss-Legendre quadrature, our strategy directly yields $b_i$ and requires solving system of  $s$ equations of \eqref{nr_plug_g} to determine $a_{i,j}$. The general approach requires solving $s^2+s$ equations stemming from the simplifying order conditions $B(s)$ and $C(s)$.

\section{Numerical experiments}\label{S-ne}

We present numerical experiments for selected newly derived implicit Runge-Kutta methods and compare errors and {\it experimental order of convergence} (EOC) with standard collocation implicit Runge-Kutta methods  with the same number of stages. 

To begin we define the errors that shall be used to analyze numerical solutions. Let 
\begin{align}\label{errors1}
e_a:=\max_{n=1,\ldots,N}|y(x_n)-y_n|,  \
e_r:=\frac{e_a}{\underset{n=1,\ldots,N}{\max}|y(x_n)|}, 
\ e_{2}:=h\sqrt{\sum_{n=1}^N|y(x_n)-y_n|^2}, \ e_{r2}:=\frac{e_2}{h\sqrt{\sum\limits_{n=1}^N|y(x_n)|^2}}
\end{align}
be the absolute and the relative errors in the maximum and the discrete $L^2$-norm, respectively. 
 Further, by
\begin{align}\label{errors2}
e_b:=\max|y(b)-y_N|,\quad e_m:=\frac{\sum\limits_{n=1}^N|y(x_n)-y_n|}{N}, \quad e_n:=\sqrt{\sum_{n=1}^N|y(x_n)-y_n|^2}
\end{align}
we denote the maximal error at the end point of the computational interval, the mean of absolute errors at each point, and the Euclidean norm of the absolute errors, respectively. 
Finally,  EOC is given by 
\begin{align}\label{EOC}
EOC_{N,2N}:=\log_2\left(\frac{e^N}{e^{2N}}\right),
\end{align}
where $e^N$ is the error computed with the discrete step size $h=1/N$ and $e^{2N}$ is the error computed with half the discrete step size.

\subsection{Experiment 1: linear stability}

Firstly, we consider standard IVP to test linear stability of Runge-Kutta methods,
\begin{align*}
y'(x)&=-15y(x), \ x \in [0,1]\\
y(0)&=1
\end{align*}
with the exact solution $y(x)=\mathrm{e}^{-15x}.$ 

In the upper left half of Table~\ref{T-e1} we present relative errors and the corresponding EOC for the following four stage stiffly accurate implicit Runge-Kutta methods with the explicit first line:  \\
nIRK4 - new IRK  based on closed Newton-Cotes formula,\\ nIRK4c - new IRK based on closed Newton-Cotes formula using Cauchy's formula,\\ sIRK4 - standard collocation IRK with $\tau_i=i/4,$ $i=0,\ldots,4$.
\\
We term them ``closed'' as they include the approximation of function values at the end points $\yn$ and $\ynn.$ 

The upper right half of Table~\ref{T-e1}  contains relative errors and EOC for the following four stage fully implicit Runge-Kutta methods:\\
nIRK4o - new IRK  based on open Newton-Cotes formula,\\ nIRK4oc  - new IRK  based on open Newton-Cotes formula using Cauchy's formula,\\ sIRK4o - standard collocation IRK with $\tau_i=i/5,$ $i=1,\ldots,4$.\\
Obviously, we refer to these methods as ``open'' because, as opposed to the ``closed'' ones, their stages do not include the end points $\xn$ and $\xnn$. 
The graphs of relative errors in logarithmic scale for the four stage implicit Runge-Kutta methods from Table~\ref{T-e1} are depicted in Figures~\ref{F-e1-4c} and \ref{F-e1-4o}.
\begin{table}
\centering
 \caption{Experiment 1: relative errors and EOC for newly derived and standard implicit Runge-Kutta methods}\label{T-e1}
\begin{tabular}{rrcc||cc||cc||cc||cc||cc}
\toprule
\multicolumn{2}{r}{}& \multicolumn{2}{c||}{nIRK4} & \multicolumn{2}{c||}{nIRK4c} & \multicolumn{2}{c||}{sIRK4} &
 \multicolumn{2}{c||}{nIRK4o} & \multicolumn{2}{c||}{nIRK4oc} & \multicolumn{2}{c}{sIRK4o}\\
\cline{3-4}\cline{5-6}\cline{7-8}\cline{9-10}\cline{11-12}\cline{13-14}
N &&  $e_r$ & EOC & $e_r$ & EOC & $e_r$ & EOC & $e_r$ & EOC & $e_r$ & EOC & $e_r$ & EOC \\
\cline{3-4}\cline{5-6}\cline{7-8}\cline{9-10}\cline{11-12}\cline{13-14}
 2 && 4.67e-02    &  - & 1.41e-01 & - & 6.50e-02
    & - & 7.24e-03 & - & 2.77e-01& - & 7.82e-02 & -
    \\
 4 && 1.75e-04   & 8.06 &  4.78e-04 & 8.20 &  3.46e-04    & 7.55
 &  9.57e-06& 9.56 & 3.77e-02 & 2.88 & 7.12e-04 & 6.78\\
 8 &&   2.04e-06
    & 6.42 &   5.45e-05 & 3.13 &  1.13e-05     & 4.94
    & 2.82e-08   & 8.41 & 1.27e-04 &  8.22   & 1.57e-05 & 5.51\\
 16 && 2.89e-08    &  6.14 &  3.09e-06 & 4.14 &  5.44e-07     & 4.38
&  1.27e-10 & 7.80 & 2.39e-06& 5.76 & 7.45e-07 & 4.39 \\ 
 32 &&   4.40e-10
   &   6.04 &    1.87e-07 & 4.04 & 3.17e-08
     & 4.10 & 1.81e-12  & 6.13 & 8.86e-08& 4.75  & 4.33e-08 & 4.10 \\
 64 &&  6.84e-12
  & 6.01 &   1.16e-08 & 4.01 &  1.94e-09     & 4.03 & 8.30e-14 & 4.44 & 4.20e-09& 4.40& 2.66e-09 & 4.03 \\
 128 &&   1.07e-13
    &  6.00 &  7.25e-10 & 4.00 &  1.21e-10
     & 4.01 & 5.00e-15 & 4.08 & 2.26e-10& 4.21  & 1.66e-10 & 4.00\\
 \bottomrule
 \end{tabular}
 \begin{tabular}{rrcc||cc||cc||cc||cc||cc}
\toprule
\multicolumn{2}{r}{}& \multicolumn{2}{c||}{nIRK5} & \multicolumn{2}{c||}{nIRK5c} & \multicolumn{2}{c||}{sIRK5}
& \multicolumn{2}{c||}{nIRK3o} & \multicolumn{2}{c||}{nIRK3oc} & \multicolumn{2}{c}{sIRK3o}\\
\cline{3-4}\cline{5-6}\cline{7-8}\cline{9-10}\cline{11-12}\cline{13-14}
N &&  $e_r$ & EOC & $e_r$ & EOC & $e_r$ & EOC & $e_r$ & EOC & $e_r$ & EOC & $e_r$ & EOC \\
\cline{3-4}\cline{5-6}\cline{7-8}\cline{9-10}\cline{11-12}\cline{13-14}
 2 &&     7.24e-03 &  - & 4.33e-02 & - &  1.96e-02    & -&    4.67e-02  &  - &  1.63e-01 & - &   1.63e-01   & -\\
 4 & & 9.55e-06 & 9.56 &  2.04e-04 & 7.73 &  6.79e-05    & 8.18&    1.75e-04 & 8.06 &  2.81e-04 & 9.18 &  2.81e-04   & 9.19\\
 8 &&     2.76e-08    & 8.43 &  2.09e-06 & 6.60 &  6.56e-07     & 6.70&  2.04e-06   & 6.42 &   6.95e-05 & 2.01 &  6.95e-05     & 2.01\\
 16 &&    1.00e-10 & 8.11  &  2.99e-08& 6.12 & 9.09e-09     & 6.17&   2.89e-08  & 6.14  & 4.04e-06 & 4.11 &  4.04e-06    & 4.11\\  
 32 &&    3.83e-13  & 8.03   &   4.58e-10 & 6.03 &  1.38e-10     & 6.04&   4.40e-10  &  6.04  &   2.46e-07 & 4.04 &  2.46e-07     & 4.04\\
 64 && 1.00e-15 & 8.01&    7.12e-12& 6.01 &  2.14e-12    & 6.01&  6.84e-12 &     6.01 & 1.52e-08 & 4.01 &  1.52e-08    & 4.01\\
 128 &&   0.00     & 8.02  & 1.11e-13 & 6.00 &  3.30e-14     & 6.00&  1.07e-13    & 6.00  & 9.51e-10 & 4.00 & 9.50e-10     & 4.00\\
  \bottomrule
 \end{tabular}
\end{table}

In the lower left half of Table~\ref{T-e1} and Figure~\ref{F-e1-5} we present the results obtained  by selected five stage stiffly accurate implicit Runge-Kutta methods with the explicit first line:  \\
nIRK5 - new IRK  based on closed Newton-Cotes formula,\\ nIRK5c - new IRK based on closed Newton-Cotes formula using Cauchy's formula,\\ 
sIRK5 - standard collocation IRK with $\tau_i=i/5,$ $i=0,\ldots,5$.

Finally, the lower right half of Table~\ref{T-e1} and Figure~\ref{F-e1-3} contain the relative errors and EOC for selected three stage fully implicit Runge-Kutta methods:\\
nIRK3o - new IRK  based on open Newton-Cotes formula,\\ nIRK3oc  - new IRK  based on open Newton-Cotes formula using Cauchy's formula,\\ 
sIRK3o - standard collocation IRK with $\tau_i=i/4,$ $i=1,\ldots,3$.  
\begin{figure}
\centering
\begin{subfigure}[b]{0.49\textwidth}
\includegraphics[scale=0.56]{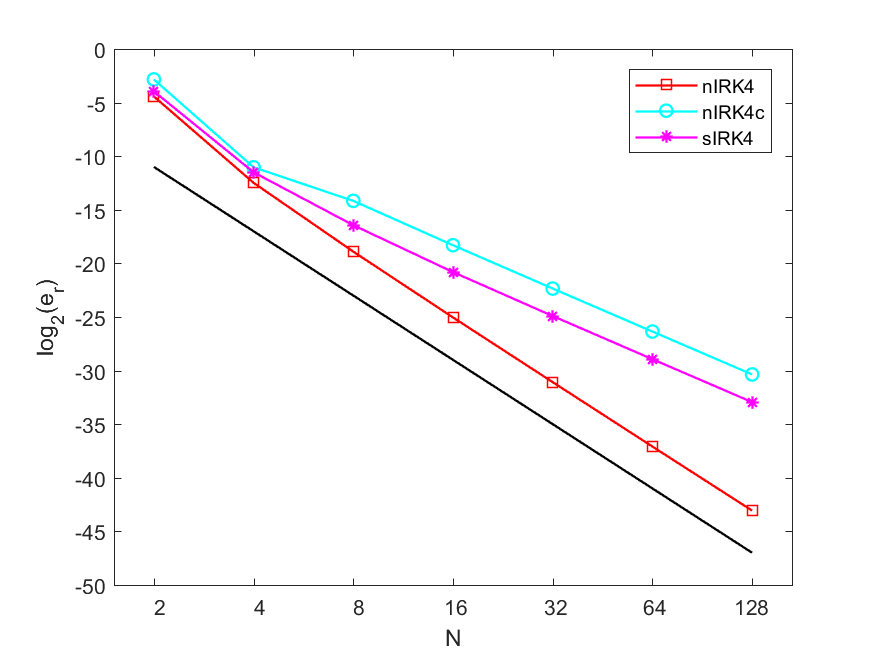}
\caption{solid line represents reference slope $N^{-6}$}\label{F-e1-4c}
\end{subfigure}
\begin{subfigure}[b]{0.49\textwidth}
\includegraphics[scale=0.56]{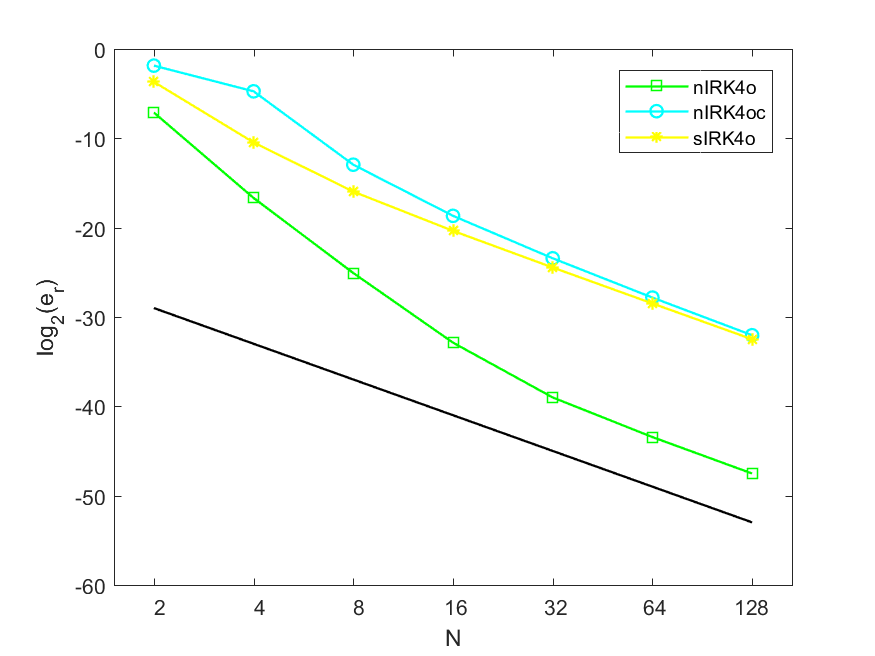}
\caption{solid line represents reference slope $N^{-4}$}\label{F-e1-4o}
\end{subfigure}
\begin{subfigure}[b]{0.49\textwidth}
\includegraphics[scale=0.56]{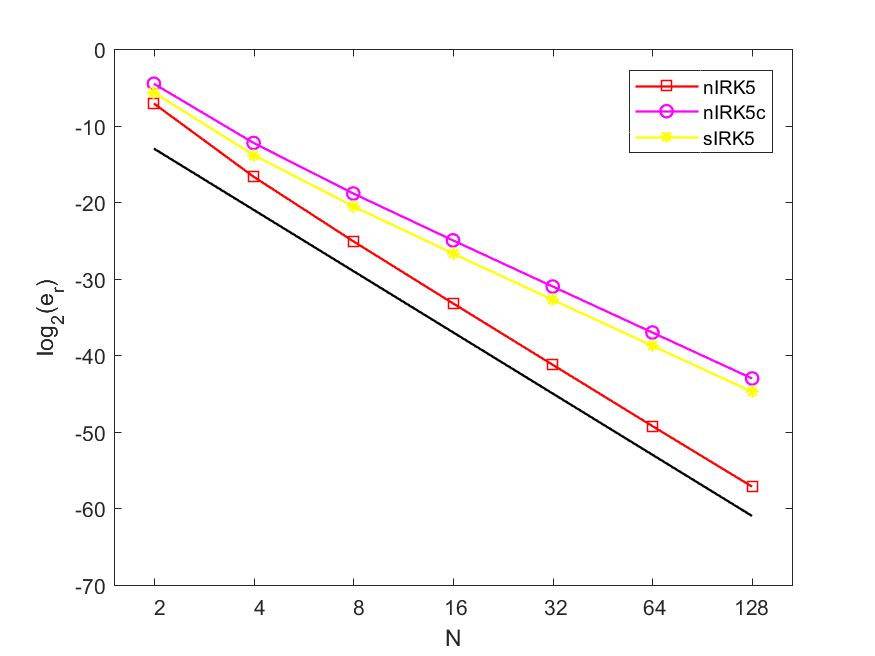}
\caption{solid line represents reference slope $N^{-8}$}\label{F-e1-5}
\end{subfigure}
\begin{subfigure}[b]{0.49\textwidth}
\includegraphics[scale=0.56]{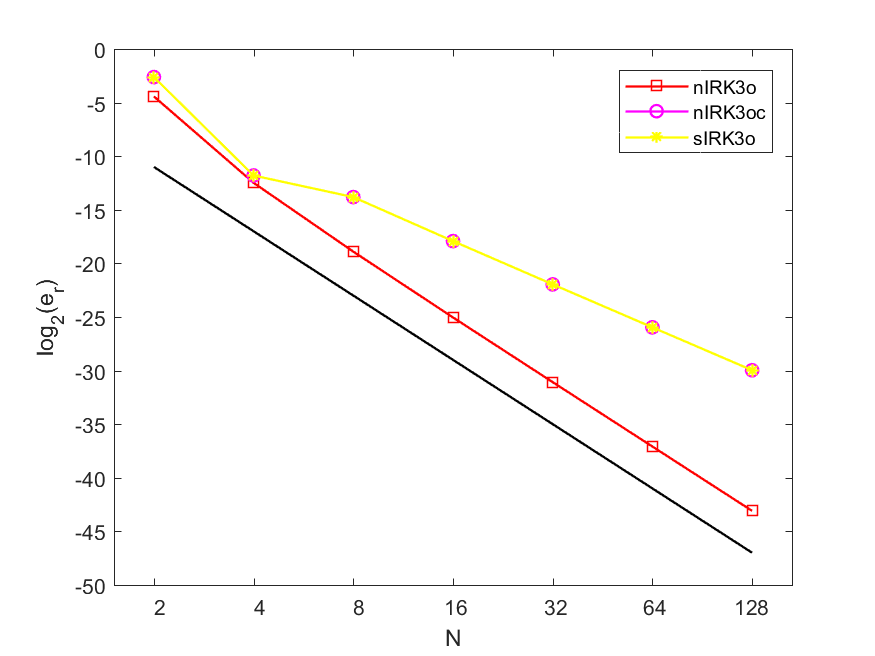}
\caption{solid line represents reference slope $N^{-6}$}\label{F-e1-3}
\end{subfigure}
\caption{Experiment 1: relative errors for newly derived and standard implicit Runge-Kutta methods}
\end{figure}

Note, in particular, the EOC of our newly derived methods nIRK4, nIRK5 and nIRK3o is 6, 8 and 6, respectively, whereas the theoretical orders of convergence are 4, 6 and 4, respectively.  The reason in the case of nIRK4 might be that this method satisfies 13 out of 16 order conditions for $p=6$ corresponding to linear ODEs from Table~\ref{T-oc-lin}. 
Similarly, nIRK3o satisfies 11 out of 16 order conditions from Table~\ref{T-oc-lin}.
The methods derived by  Cauchy's repeated integration formula \eqref{CF} yield the highest relative errors which  is in accordance with the theory. Indeed,  using \eqref{CF} to compute approximations of repeated integrals by Newton-Cotes formulas may cause a loss of precision, cf. \cite{NT}.
   
\subsection{Experiment 2: stiff scalar linear equations}

Solving stiff equations is difficult for many numerical methods. We test our newly derived implicit Runge-Kutta methods on the stiff initial value problems  and compare the results  with the standard implicit Runge-Kutta methods with the same number of stages. 
\begin{itemize}
\item[{\bf (2A)}] To begin, we consider a moderately stiff IVP:
\begin{align*}
y'(x)&=-100y(x)+99\mathrm{e}^{2x}, \ x \in [0,0.5]\\
y(0)&=0
\end{align*}
with the exact solution $\displaystyle y(x)=\frac{33}{34}(\mathrm{e}^{2x}-\mathrm{e}^{-100x}),$ see e.g. \cite{ahmad,ramos,7-10}.
\end{itemize}
We plot the relative errors  computed for the values $N=2^{k},$ $k=2,\ldots,10$. 
Figure~\ref{F-e2A} indicates good performance of our newly derived implicit Runge-Kutta methods with four and three stages. Their relative errors are  smaller than those of corresponding standard collocation methods. They all converge with EOC=4. Note, in particular, the performance of nIRK4 which is almost identical with the well-known LobattoIIIA with four stages (nIRK-L4), from now on referred to as Lobatto4, at least for  $N< 2^7$ with $e_r\approx 4e-11.$  These two methods, nIRK4 and Lobatto4, share the same number of stages and computational costs, but differ in the theoretical order of convergence, that is 4 and 6, respectively. 
\begin{figure}[!h]
\centering
\begin{subfigure}[b]{0.49\textwidth}
\includegraphics[scale=0.56]{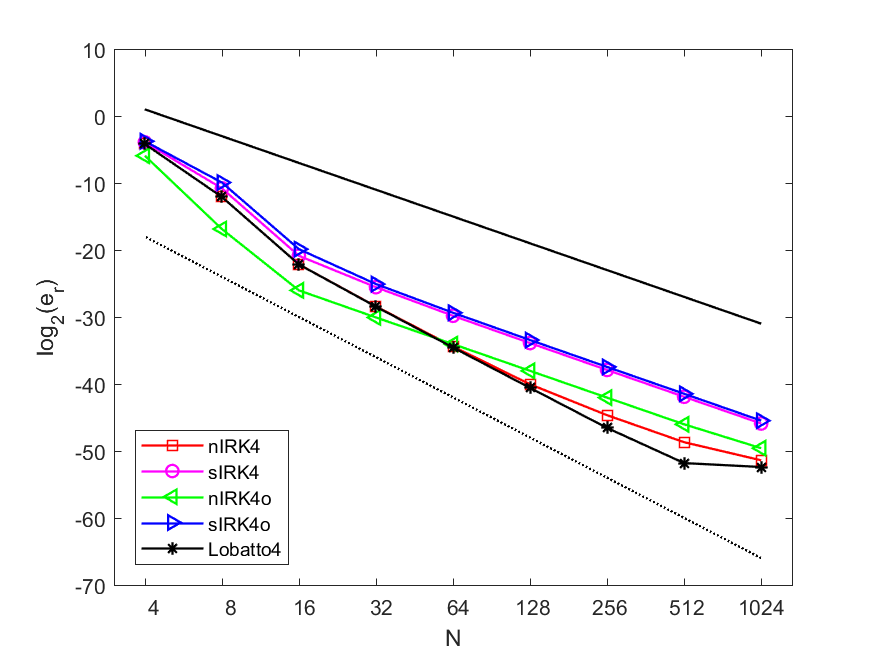}
\caption{four stage methods}
\end{subfigure}
\begin{subfigure}[b]{0.49\textwidth}
\includegraphics[scale=0.56]{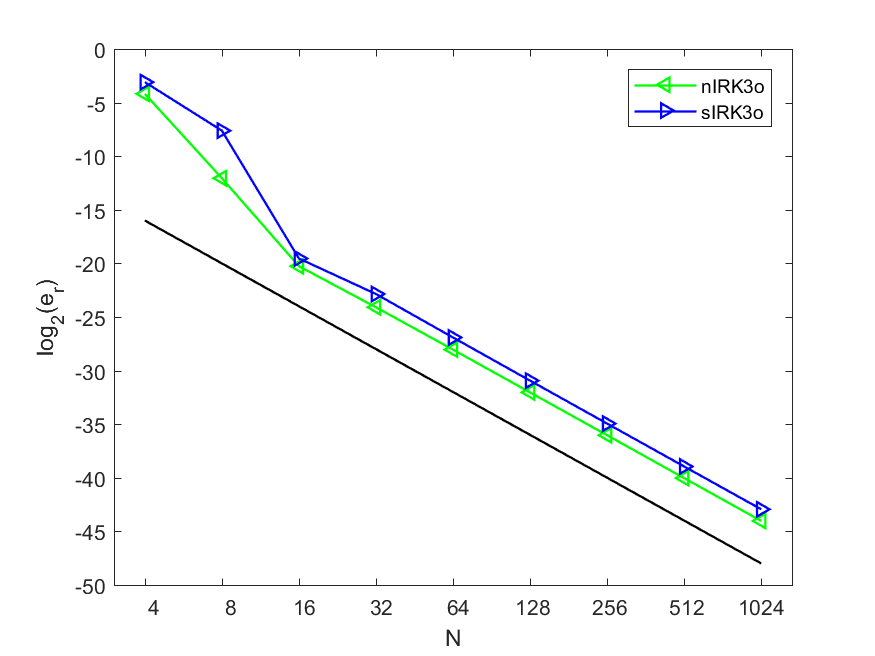}
\caption{three stage methods}
\end{subfigure}
\caption{Experiment (2A): relative errors; solid and dotted lines represent the reference slope  $N^{-4}$ and $N^{-6},$ respectively}\label{F-e2A}
\end{figure}

\begin{itemize}
\item[{\bf (2B)}] Now we consider a highly stiff IVP:
\begin{align*}
y'(x)&=-1000y(x)+\mathrm{e}^{-2x}, \ x \in [0,10]\\
y(0)&=0
\end{align*}
with the exact  solution $\displaystyle y(x)=\frac{1}{998}(\mathrm{e}^{-2x}-\mathrm{e}^{-1000x}),$ see, e.g. \cite{ramosRK,highly}.
\end{itemize}
The graphs of numerical solutions for $N=256$ computed by chosen four and five stage methods are depicted in Figure~\ref{F-e2B}. 
In the case of ``closed'' four stage methods we can observe small oscillations in approximate solutions which is not the case for ``open'' four stage methods, see Figure~\ref{F-e2B-4s}, and ``closed'' five stage methods, see Figure~\ref{F-e2B-5s}. Note that nIRK4 and sIRK4 produce almost identical approximations while nIRK4o and nIRK5 yield better approximations than their standard collocation counterparts sIRK4o and sIRK5, respectively. Recall that our ``open'' methods are derived using the open Newton-Cotes quadrature formulas excluding the endpoints of the interval, cf. Subsection~\ref{SS-oNC}. 
Table~\ref{T-e2B} contains the errors as defined in \eqref{errors1} and \eqref{errors2} together with the absolute error at the point $x=0.390625$ ($e_x$). For comparison we have added the corresponding errors for the fifth-order method introduced and analyzed in \cite{highly}.  The smallest values of errors $e_a,$ $e_x,$ $e_m,$ $e_n$ among chosen four and five stage methods are obtained by nIRK4o (fourth-order) and nIRK5 (sixth-order). The standard collocation IRK only exhibit smaller absolute error $e_b$ at the end point $b=10.$ Note again that our fourth-order method nIRK4 produces almost identical errors as the sixth-order Lobatto4.  
\begin{figure}
\centering
\begin{subfigure}[b]{0.49\textwidth}
\includegraphics[scale=0.56]{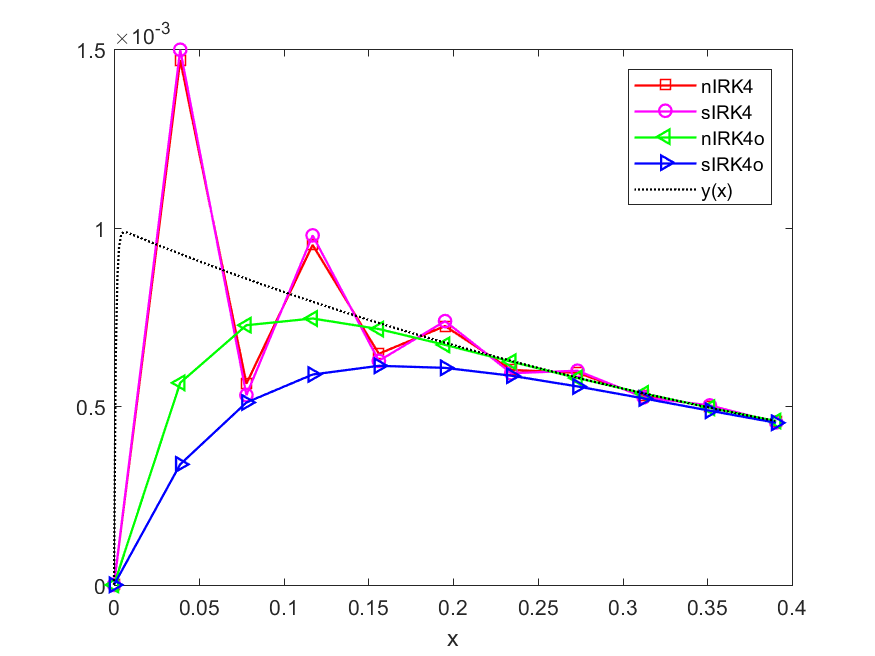}
\caption{four stage methods}\label{F-e2B-4s}
\end{subfigure}
\begin{subfigure}[b]{0.49\textwidth}
\includegraphics[scale=0.56]{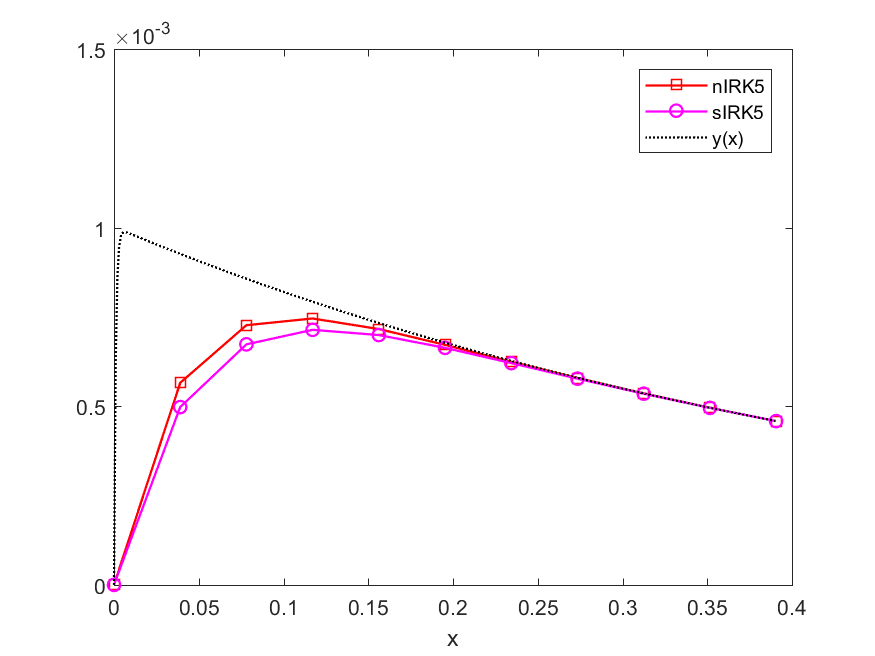}
\caption{five stage methods}\label{F-e2B-5s}
\end{subfigure}
\caption{Experiment (2B):  approximations of exact solution for $N=256$}\label{F-e2B}
\end{figure}
\begin{table}[!h]
\centering
\caption{Experiment (2B):  errors of numerical approximations computed for $N=256$} \label{T-e2B}
\begin{tabular}{rrc||c||c||c||c||c||c||c}
\toprule
 && {nIRK4}& {sIRK4}& {nIRK4o} & {sIRK4o}& {Lobatto4} & (11) in \cite{highly}& {nIRK5}&  {sIRK5}\\
\cline{3-3}\cline{4-4}\cline{5-5}\cline{6-6}\cline{7-7}\cline{8-8}\cline{9-9}\cline{10-10}
 $e_a$& & 5.43e-04&   5.71e-04& 3.61e-04&  5.89e-04&  5.43e-04& 6.64e-04 & 3.61e-04 &  4.29e-04\\
 $e_x$ && 2.17e-06&  3.64e-06 & 3.66e-08 &     4.95e-06 & 2.17e-06& -& 3.69e-08 &    2.07e-07\\
 $e_b$ && 3.15e-22&   1.40e-22 &  1.68e-18&   4.31e-21& 1.25e-22& 1.50e-12 & 1.96e-23 &  2.03e-24 \\
 $e_m$ & & 4.60e-06 &  5.17e-06 & 2.20e-06  &  5.57e-06 & 4.60e-06& 3.47e-05 &2.20e-06  &2.92e-06\\
 $e_n$ & &6.45e-04 &  6.96e-04 & 3.87e-04 & 7.28e-04  &6.45e-04& 1.76e-03 & 3.87e-04  & 4.75e-04\\
\bottomrule
\end{tabular}
\end{table}

\subsection{Experiment 3: nonlinear scalar stiff equation}
In this experiment we consider a simple model of flame propagation, 
\begin{equation*}
\begin{aligned}
y'(x)&=y^2(x)-y^3(x), \ x \in \left[0,\frac{2}{\delta}\right] \\
y(0)&=\delta,
\end{aligned}
\end{equation*}
where $\delta <<1.$ 
The exact solution is given by 
\begin{align*}
y(x)=\frac{1}{W(A{\rm e}^{A-x})+1},
\end{align*}
where $A=1/\delta-1$ and $W(t)$ is the so-called Lambert W function  \cite{W}. It solves the equation $W(t){\rm e}^{W(t)}=t.$ 
\begin{figure}[!h]
\centering
\begin{subfigure}[t]{0.49\textwidth}
\includegraphics[scale=0.56]{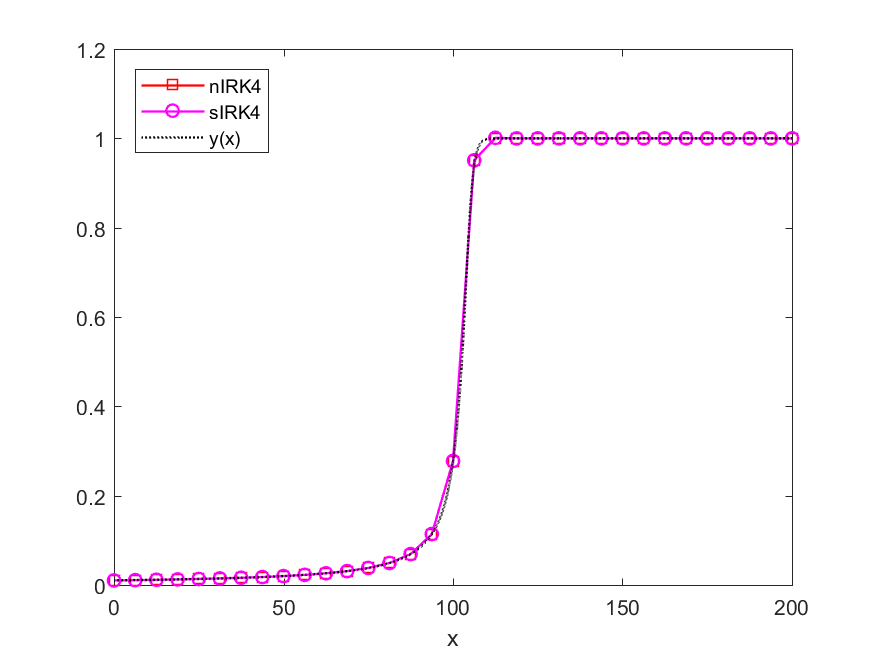}
\caption{numerical solutions for $N=32$ }\label{F-e3-sol}
\end{subfigure}
\begin{subfigure}[t]{0.49\textwidth}
\includegraphics[scale=0.56]{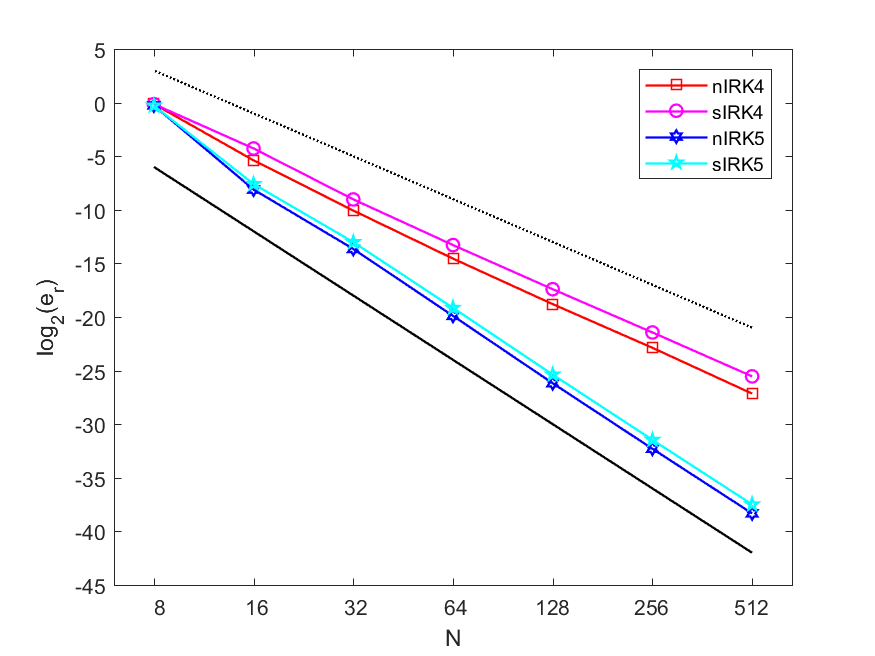}
\caption{relative errors; solid and dotted lines represent reference slopes $N^{-6}$ and $N^{-4},$ respectively}\label{F-e3-EOC}
\end{subfigure}
\caption{Experiment 3: numerical solutions and relative errors }\label{F-e3}
\end{figure}
We compute the above IVP by nIRK4, sIRK4 and nIRK5, sIRK5 for $\delta=0.01$.  Figure~\ref{F-e3-sol} depicts graphs of exact and numerical solutions computed for $N=32.$ The relative errors are plotted in Figure~\ref{F-e3-EOC}.   We can observe smaller relative errors of our newly derived IRK compared to standard collocation IRK. The EOC confirms theoretical order of convergence, $p=4$ for selected four stage  methods and $p=6$ for selected five stage methods.

\subsection{Experiment 4: Prothero--Robinson problem}

All newly derived schemes are A-stable which, in general, guarantees a stable numerical approximation.  The Prothero-Robinson example \cite{PRS} reveals the so-called {\em order reduction phenomenon for stiff problems} which is present for one-step methods. For instance, for fully implicit
Runge-Kutta methods like Gauss-Legendre methods, EOC decreases from $2s$ to $s$ (number of stages), cf. \cite{rangCS}. We refer the reader to \cite{HaWa,skvortsovA,rang,rangCS,PRS}, among others, for the role of stage order in solving stiff problems.

We thus proceed our experimental analysis with a linear Prothero-Robinson problem 
\begin{align*}
y'(x)&=\lambda (y(x)-\varphi(x))+\varphi'(x), \quad x\in [0,b], \ b >0\\
y(0)&=\varphi(0),
\end{align*}
   where $\lambda<<0$ and $\varphi$ is a given function and the exact solution at the same time. This problem  has been widely used to analyze and design numerical methods for solving stiff problems, see e.g. \cite{rang,skvortsovA}.
   \\
 Let us consider the first case:
   \begin{itemize}
   \item[\bf (4A)] $\lambda=-10^{6},$ $\varphi(x)=\sin\left(\frac{\pi}{4}+x\right),$ $y(0)=\varphi(0)=\frac{\sqrt{2}}{2},$ $x\in[0,15]$  
     \end{itemize}
To begin we compute the Prothero-Robinson problem with the parameters specified in (4A) for $N=2^k,$ $k=2,\ldots,9$.  Figure~\ref{F-e4A-all} shows the resulting relative errors (in discrete $L^2$-norm) and indicates the corresponding EOC for chosen methods. We can clearly observe that  nIRK5, sIRK5, Lobatto4, nIRK4, sIRK4, sIRK4o, nIRK3o, sIRK3o converge with EOC=4, while nIRK5c, nIRK4c and nIRK3oc converge with EOC=2.  Note that nIRK4o exhibit convergence of order 4 for large step sizes ($N<32$) and fail to converge for smaller step sizes. The worst results are obtained by nIRK4oc. This is in accordance with the theory since the stage order of this method is 0.  The convergence failure of nIRK4oc can be also seen in Figure~\ref{F-e4A-sol}, where  the graphs of approximations computed for $N=8$ by nIRK4oc and nIRK3o are plotted.
\begin{figure}
\centering
\begin{subfigure}[b]{0.49\textwidth}
\includegraphics[scale=0.56]{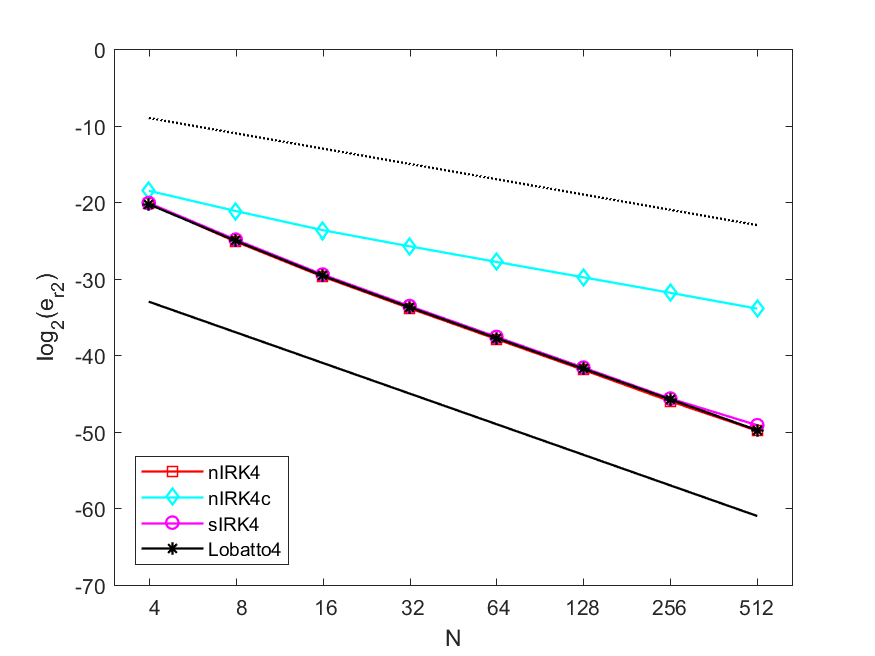}
\caption{``closed'' four stage methods}\label{F-e4A-4c}
\end{subfigure}
\begin{subfigure}[b]{0.49\textwidth}
\includegraphics[scale=0.56]{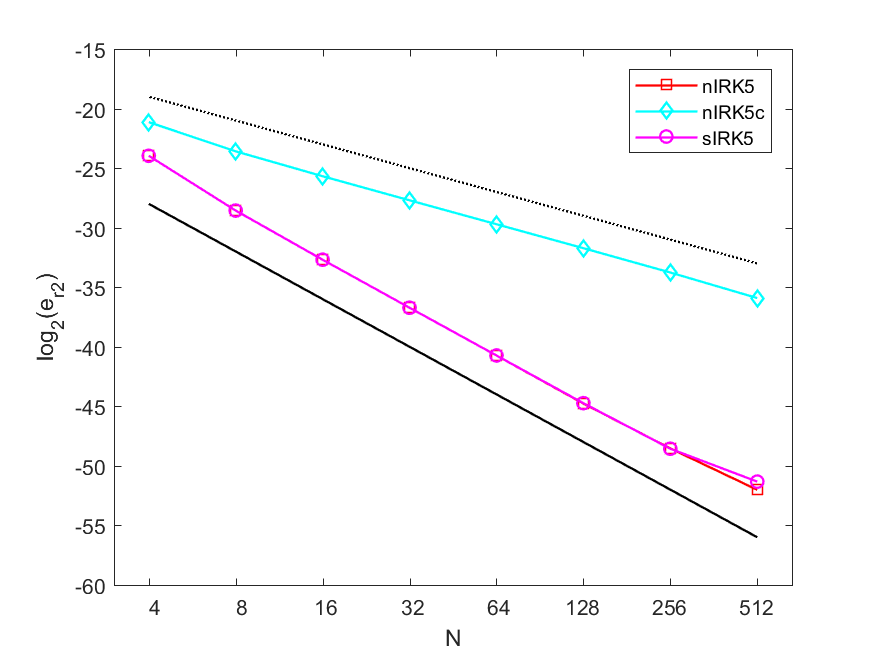}
\caption{``closed'' five stage methods}\label{F-e4A-5}
\end{subfigure}
\begin{subfigure}[b]{0.49\textwidth}
\includegraphics[scale=0.56]{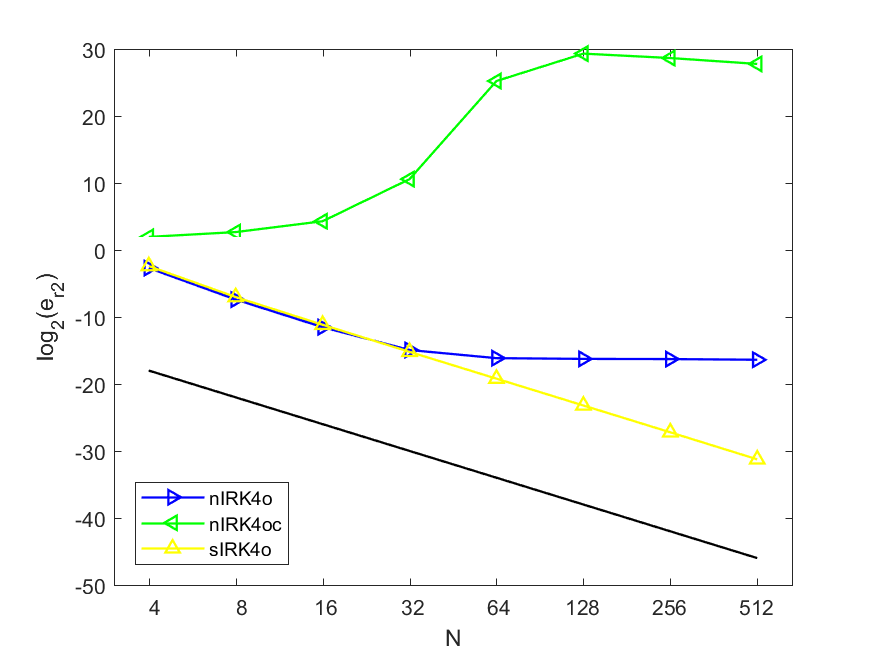}
\caption{``open'' four stage methods}\label{F-e4A-4o}
\end{subfigure}
\begin{subfigure}[b]{0.49\textwidth}
\includegraphics[scale=0.56]{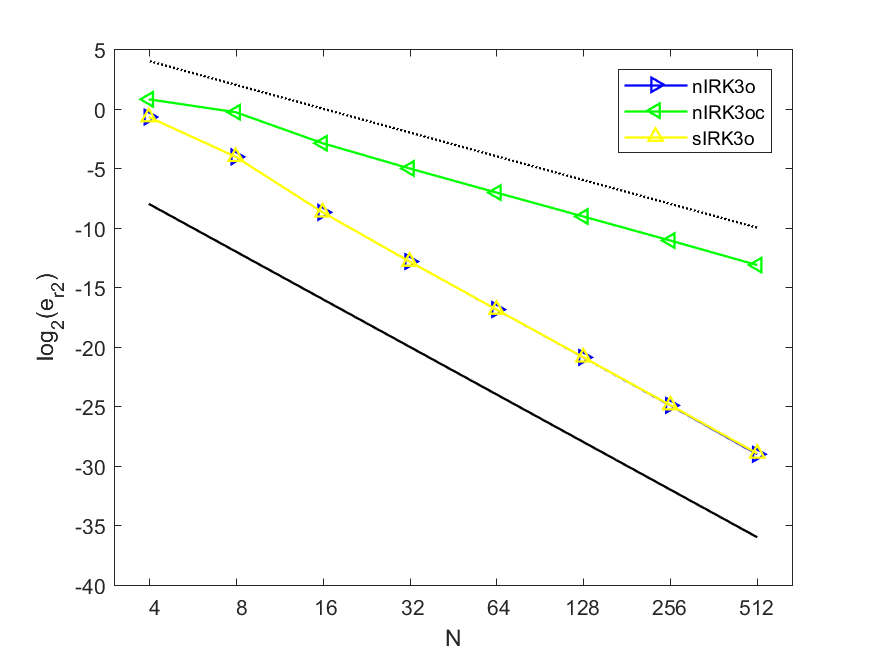}
\caption{``open'' three stage methods}\label{F-e4A}
\end{subfigure}
\caption{Experiment (4A): relative errors; solid and dotted lines represent the reference slopes $N^{-4}$ and $N^{-2}$, respectively}\label{F-e4A-all}
\end{figure}
We further plot the dependence of EOC on the stiffness parameter $-\lambda$ ranging from $1$ to $10^6.$  See Figure~\ref{F-e4A-EOC}. In the case of five stage methods we observe the  order reduction phenomenon with EOC dcreasing from 6 to 4 for nIRK5, sIRK5, the same occurs for four stage Lobatto4, and order reduction from 6 to 2 is obvious for nIRK5c. 
The same phenomenon is present for nIRK4c and nIRK3oc for which EOC decreases from 4 to 2. We can conclude that nIRK4, sIRK4, nIRK3o, sIRK3o converge with EOC $\approx$ 4 for all given values of $\lambda.$
\begin{figure}
\centering
\begin{subfigure}[b]{0.49\textwidth}
\includegraphics[scale=0.56]{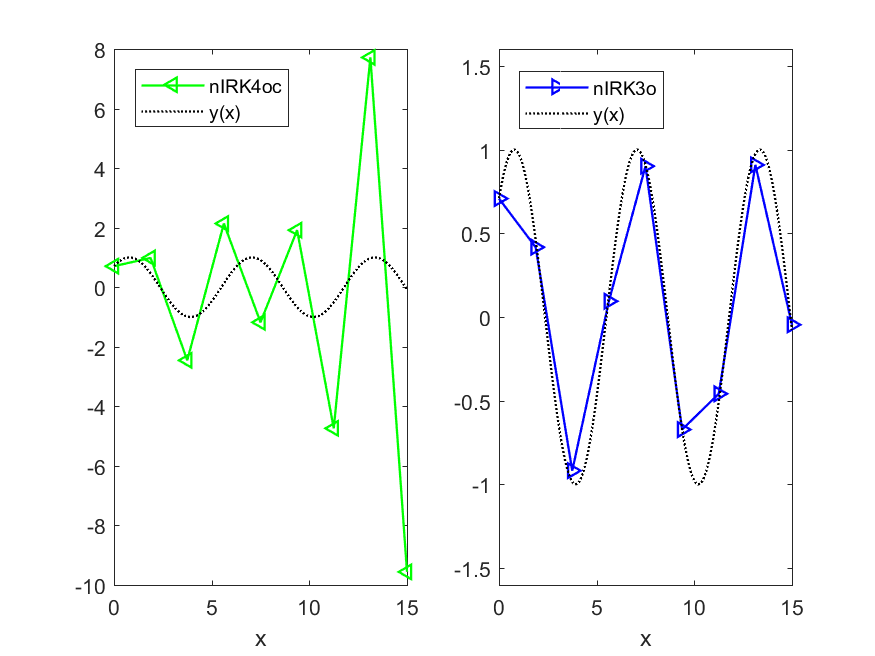}
\caption{numerical solutions for $N=8$}\label{F-e4A-sol}
\end{subfigure}
\begin{subfigure}[b]{0.49\textwidth}
\includegraphics[scale=0.56]{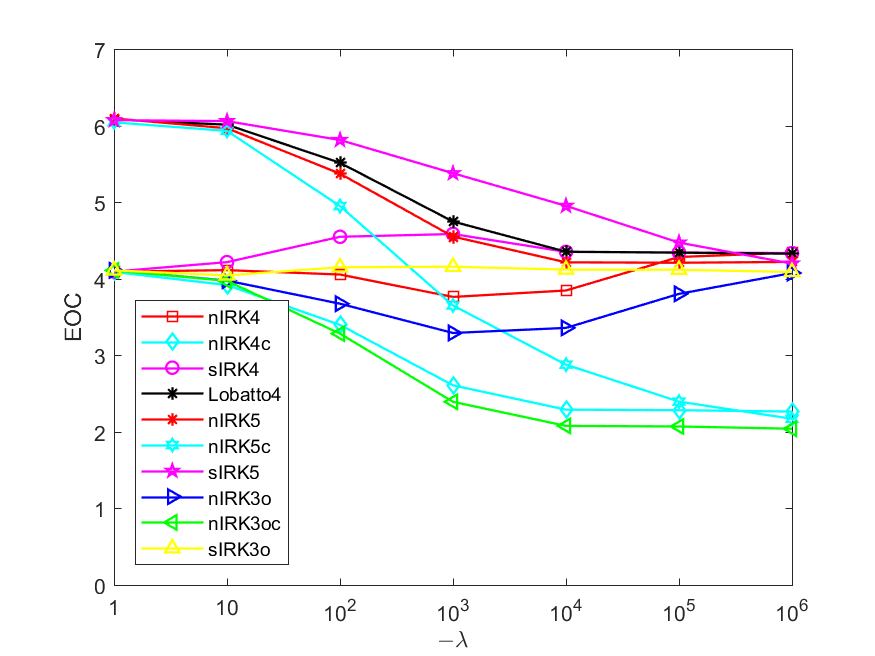}
\caption{EOC vs. stiffness parameter $-\lambda$}\label{F-e4A-EOC}
\end{subfigure}
\caption{Experiment (4A): numerical solutions and stiffness influence on EOC}\label{F-e4A}

\end{figure}

A more general form of the Prothero-Robinson problem as considered in \cite{S} prescribes the initial value $y(0)=y_0\neq \varphi(0)$ and possesses the exact solution  $y(x)=\varphi(x)+(y_0-\varphi(0))\mathrm{e}^{\lambda x}.$  The second case we consider complies with the latter form:
    \begin{itemize}
   \item[\bf (4B)] $\lambda=-200,$ $\varphi(x)=10-(10+x)\mathrm{e}^{-x},$ $y_0=10,$ $x\in[0,15].$
   \end{itemize}
Figure~\ref{F-e4B-sol} depicts the graphs of approximations to the  Prothero-Robinson problem specified in (4B) computed by selected four and five stage IRK with the step size $h=0.2.$ Note in particular, that for $h=0.2$ we have obtained the same approximate  solution  computed by Lobatto4 and nIRK4. The relative errors for ``closed'' five stage IRK are plotted in Figure~\ref{F-e4B-EOC}. It indicates EOC=6 for nIRK5c and sIRK5 while nIRK5 converges with EOC=8. 
\begin{figure}
\centering
\begin{subfigure}[t]{0.49\textwidth}
\includegraphics[scale=0.56]{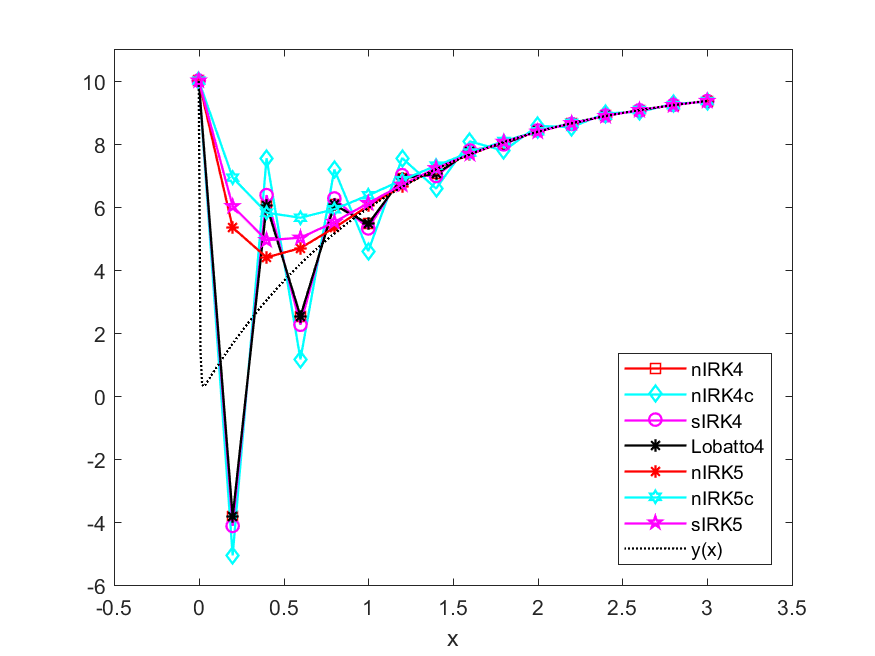}
\caption{numerical solutions for $h=0.2$}\label{F-e4B-sol}
\end{subfigure}
\begin{subfigure}[t]{0.49\textwidth}
\includegraphics[scale=0.56]{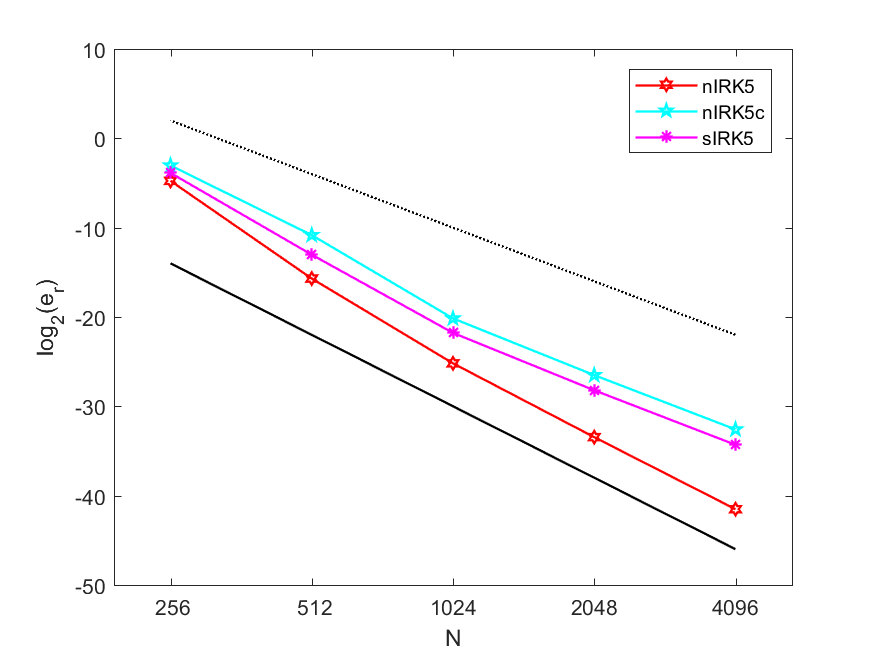}
\caption{relative errors; solid and dotted lines represent reference slopes $N^{-8}$ and $N^{-6},$ respectively}\label{F-e4B-EOC}
\end{subfigure}
\caption{Experiment (4B):  numerical solutions and relative errors}  \label{T-e4B}
\end{figure}

\subsection{Experiment 5: Kap's system}
In the fifth experiment we consider a two-dimensional  nonlinear stiff system, namely Kap's system, arising in chemical kinetics of reactions,  
\begin{align*}
y_1'(x)&=-(\mu+2)y_1+\mu y_2^2, \ x \in [0,1]\\
y_2'(x)&= y_1-y_2-y_2^2, \\
y_1(0)&=1\\
y_2(0)&=1.
\end{align*}
Its exact solution $y(x)=[y_1(x),y_2(x)]=[\mathrm{e}^{-2x},\mathrm{e}^{-x}]$ is independent of the stiffness parameter $\mu.$ For the numerical analysis of the latter system we refer the reader to, for instance, \cite{akanbi,highly,skvortsovKap,skvortsovA} and the references therein. 
We compute numerical solutions for $N=15$ and $\mu=10^k,$ $k=1,\ldots,5$ by both ``closed'' and ``open'' four stage methods. We define the maximal relative errors for a two-dimensional problem,
\begin{align*}
\varepsilon_r:=\frac{\underset{n=1,\ldots,N}{\max}\sqrt{|y_1(x_n)-y_{1,n}|^2+|y_2(x_n)-y_{2,n}|^2}}{\underset{n=1,\ldots,N}{\max}\sqrt{|y_1(x_n)|^2+|y_2(x_n)|^2}},
\end{align*}
and list the corresponding values in Table~\ref{T-e5}. It seems that only ``closed'' four stage methods result in a suitable approximation of Kap's system as given above. Note, in particular that both nIRK4 and Lobatto4 yield very similar relative errors. Keep in mind their theoretical convergence orders are 4 and 6, respectively. 
\begin{table}[!ht]
\centering
\caption{Experiment 5:  maximal relative errors for different values of $\mu$ and $N=15$} \label{T-e5}
\begin{tabular}{rrc||c||c||c||c}
\toprule
&&\multicolumn{5}{c}{$\mu$} \\
\cline{3-7}
 & &  $10^1$ & $10^2$ &$10^3$ & $10^4$ &$10^5$   \\
\cline{3-3} \cline{4-4}\cline{5-5}\cline{6-6}\cline{7-7}  
{nIRK3o} & & 0.4181   &     10056    &   885.6     &  9305.7     &  9325.7\\
{sIRK3o}& & 0.4374  &     5.9732    &   2.6611 &      19.206       & 19.083\\
{nIRK4} & & 1.5238e-08  & 1.4108e-08 &  4.3893e-11  & 6.1826e-11 &  7.0361e-12\\
{sIRK4}&  & 5.1496e-09  &  2.336e-09  & 1.1911e-09  & 1.0871e-09 &  1.0583e-09\\
{Lobatto4}& & 2.0937e-10   &1.4140e-09 &  5.4434e-10  & 6.4220e-11 &  6.5949e-12\\
nIRK5 & &      0.8771 &      1.2711   &    1.6602      & 1.6206     &  1.6142\\
sIRK5 &&      0.8760  &     1.2836   &    1.6415      & 1.5592    &   1.5479\\
\bottomrule
\end{tabular}
\end{table}
 In Figure~\ref{F-e5-sol} we plot the graph of numerical solution computed by nIRK4 for $\mu=1000$ and $N=15.$ The  relative errors for $\mu=1000$ and different values of $N$ obtained by ``closed'' four stage methods are depicted in Figure~\ref{F-e5-EOC}. We can observe EOC between 3 and 4 for nIRK4 and sIRK4 while Lobatto4 fails to  converge for larger values of $N$.
\begin{figure}
\centering
\begin{subfigure}[t]{0.49\textwidth}
\includegraphics[scale=0.56]{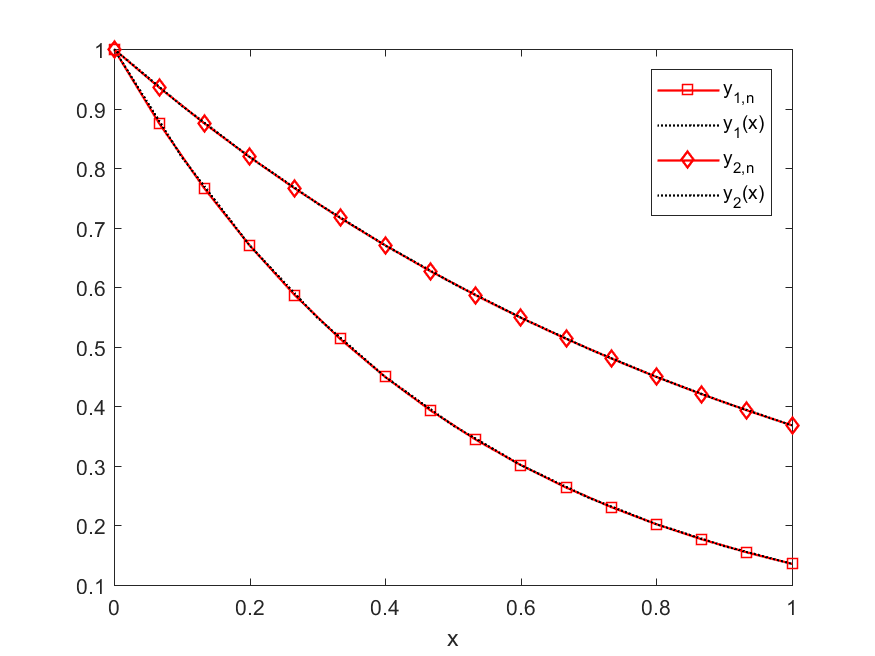}
\caption{numerical solution by nIRK4 for $N=15$}\label{F-e5-sol}
\end{subfigure}
\begin{subfigure}[t]{0.49\textwidth}
\includegraphics[scale=0.56]{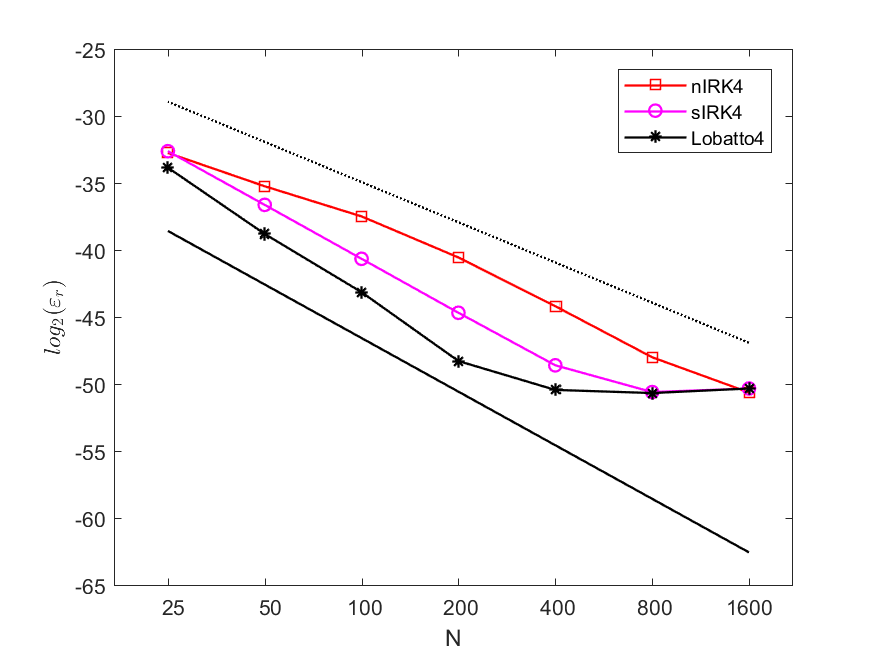}
\caption{relative errors; solid and dotted lines represent reference slope $N^{-4}$ and $N^{-3},$ respectively}\label{F-e5-EOC}
\end{subfigure}
\caption{Experiment 5: numerical solution and relative errors for $\mu=1000$}\label{F-e5}
\end{figure}

\section*{Conclusions}
We have proposed a new way of deriving coefficients of implicit Runge-Kutta methods used for numerical  approximation of initial value problems \eqref{ivp}. Instead of standard collocation approach or the general approach based on simplifying order conditions as described in Subsection~\ref{SS-RK} we have imposed  identities for  the  moments $(\xnn-x)^ky'(x)=(\xnn-x)^kf(x,y(x)),$ $k=0,1,\ldots,s$. The number of resulting equations containing repeated integrals depends on the number of stages $s$ and the numerical quadrature being used, see Section~\ref{S-new}.
We have derived several Butcher's tableaux based on closed and open modified Newton-Cotes quadrature formulas with and without employing Cauchy's repeated integration formula, as well as  based on Gauss-Legendre, -Radau and -Lobatto quadrature formulas. Various numerical experiments including stiff initial value problems for selected newly derived implicit Runge-Kutta methods were presented, among them the well-known Prothero-Robinson problem. We can conclude that the performance of newly derived methods in  terms of the order of convergence and relative errors was at least comparable or slightly better than that of the standard collocation implicit Runge-Kutta methods with the same number of stages.  Especially, recall that nIRKs for linear IVP  in Experiment 1 converged with EOC $=2s-2$  which is typical for Lobatto methods. While nIRK2 and nIRK3 indeed coincide with Lobatto IIIA methods, nIRK4 and nIRK5 {\it do not} belong to the family of Lobatto methods. 
Note in particular, that the performance in terms of relative errors and EOC of nIRK4  in the presented experiments was at least as good as or better than the performance of LobattoIIIA with the same number of stages and computational costs (Lobatto4) while having the theoretical convergence order 4 as opposed to 6 of Lobatto4. 

\subsubsection*{Acknowledgements}
The work of Hana Mizerov\' a was supported in part by the Slovak Research
and Development Agency under the contract No. APVV-23-0039 and by the grants VEGA 1/0084/23 and VEGA 1/0709/24. The work of Katar\' ina Tvrd\' a was supported by the grants KEGA 030STU-4/2023 and VEGA 1/0155/23.

\appendix
\setcounter{table}{0}
\renewcommand{\thetable}{A\arabic{table}}
\renewcommand{\thesubsubsection}{A\arabic{subsubsection}}
\section*{Appendix}
We provide the list of newly derived  as well as of some well-known implicit Runge-Kutta methods. The tables have the following columns: name/type of method, Butcher's tableau, theoretical order of convergence and values of $p,$ $q,$ $r$ corresponding to simplifying order conditions \eqref{sor}  satisfied by the underlying methods.

\subsubsection{List of new IRK  based on closed Newton-Cotes formulas}

Table~\ref{T-NCclosed} contains new IRK  derived by modified closed Newton-Cotes formulas, cf. Subsection~\ref{SS-cNC} for derivation of nIRK$s$ and nIRK$s$c. 
There is an additional column with the corresponding stability functions $R(z)$ in Table~\ref{T-NCclosed}.
\begin{table}[!ht]
\centering
\caption{nIRK$s$ and nIRK$s$c based on closed Newton-Cotes formulas for repeated integrals}\label{T-NCclosed}
\begin{tabular}[b]{cclccc}
\toprule
name & $s$ &  Butcher's tableau & $p$  & $p$ $q$ $r$ & stability function\\
\midrule
 nIRK4 & 4 & \parbox[c]{0.4\hsize}{
$\begin{array}[b]{c|cccc}
0 & 0 & 0 & 0 & 0 \\
\frac 1 3 & \frac{47}{360} & \frac{89}{360} & -\frac{19}{360} & \frac{1}{120} \\
\frac 2 3 & \frac{7}{60} & \frac{77}{180} & \frac{23}{180} & -\frac{1}{180} \\
1 & \frac 1 8 & \frac 3 8 & \frac 3 8 & \frac 1 8 \\
\hline 
 & \frac 1 8 & \frac 3 8 & \frac 3 8 & \frac 1 8\\
\end{array}
$}
& 4 & 4 3 0 & $\displaystyle -\frac{\frac{z^3}{120} + \frac{z^2}{10} + \frac{z}{2} + 1}{\frac{z^3}{120} - \frac{z^2}{10} + \frac{z}{2} - 1}$ \\
\midrule
 nIRK4c & 4 & \parbox[c]{0.4\hsize}{
$\begin{array}[b]{c|cccc}
0 & 0 & 0 & 0 & 0 \\
\frac 1 3 & \frac{1}{24} & \frac{5}{24} & -\frac{1}{4} & \frac{1}{24} \\
\frac 2 3 & \frac{1}{12} & \frac{5}{12} & \frac{1}{4} & -\frac{1}{12} \\
1 & \frac 1 8 & \frac 3 8 & \frac 3 8 & \frac 1 8 \\
\hline 
 & \frac 1 8 & \frac 3 8 & \frac 3 8 & \frac 1 8\\
\end{array}
$}
& 4 & 4 2 2 & $\displaystyle -\frac{\frac{z^3}{72} +\frac{z^2}{9} +\frac{z}{2} +1}{\frac{z^3}{72} - \frac{z^2}{9} + \frac{z}{2} - 1}$ \\
\midrule
 nIRK5 &  5 & \parbox[c]{0.4\hsize}{
$\begin{array}[b]{c|ccccc}
0 & 0 & 0 & 0 & 0&0 \\
\frac 1 4 & \frac{200}{2183}    &   \frac{429}{2078}  &    -\frac{109}{1680}    &   \frac{191}{10080} &     -\frac{43}{20160} \\
\frac 1 2 &  \frac{29}{360}     &   \frac{31}{90}    &       \frac{1}{15}    &       \frac{1}{90}    &      -\frac{1}{360}  \\
\frac 3 4 &    \frac{179}{2240}   &  \frac{377}{1120}   &   \frac{111}{560}   &     \frac{167}{1120}   &    -\frac{31}{2240} \\ 
1 & \frac{7}{90}       &   \frac{16}{45 }     &     \frac{2}{15}      &   \frac{16}{45}     &   \frac{7}{90} \\
\hline 
 &\frac{7}{90}       &   \frac{16}{45 }     &     \frac{2}{15}      &   \frac{16}{45}     &   \frac{7}{90} 
\end{array}
$}
& 6 & 6 4 1 & $\displaystyle \frac{\frac{z^4}{1680} + \frac{z^3}{84} + \frac{3z^2}{28} + \frac{z}{2} + 1}{\frac{z^4}{1680} - \frac{z^3}{84} + \frac{3z^2}{28} - \frac{z}{2} + 1}$ \\
\midrule
 nIRK5c &  5 & \parbox[c]{0.4\hsize}{
$\begin{array}[b]{c|ccccc}
0 & 0 & 0 & 0 & 0&0 \\
\frac 1 4 & \frac{371}{2880}    &   \frac{79}{720}  &    \frac{1}{480}    &   \frac{19}{720} &     -\frac{49}{2880} \\
\frac 1 2 &  -\frac{7}{120}     &   \frac{28}{45}    &       \frac{1}{15}    &   -    \frac{14}{15}    &   \frac{49}{360}  \\
\frac 3 4 &    \frac{91}{960}   &  \frac{79}{240}   &   \frac{21}{160}   &     \frac{59}{240}   &    -\frac{49}{960} \\ 
1 & \frac{7}{90}       &   \frac{16}{45 }     &     \frac{2}{15}      &   \frac{16}{45}     &   \frac{7}{90} \\
\hline 
 &\frac{7}{90}       &   \frac{16}{45 }     &     \frac{2}{15}      &   \frac{16}{45}     &   \frac{7}{90} 
\end{array}
$}
& 6 & 6 3 3 & $\displaystyle \frac{\frac{7z^4}{5760} + \frac{z^3}{64} + \frac{11z^2}{96} + \frac{z}{2} + 1}{\frac{7z^4}{5760} - \frac{z^3}{64} + \frac{11z^2}{96} - \frac{z}{2} + 1}$ \\
\bottomrule
\end{tabular}
\end{table}

\subsubsection{List of new IRK based on open Newton-Cotes formulas}

Table~\ref{T-NCopen} contains new IRK derived by modified open Newton-Cotes formulas, cf. Subsection~\ref{SS-oNC} for derivation of nIRK$s$o and nIRK$s$oc. 
\begin{table}[!ht]
\caption{nIRK$s$o and nIRK$s$oc based on open Newton-Cotes formulas for repeated integrals}\label{T-NCopen}
\centering
\begin{tabular}[b]{cclcc}
\toprule
name & $s$ & Butcher's tableau & $p$ & $p$ $q$ $r$ \\
\midrule
nIRK3o & 3 & \parbox[c]{0.65\hsize}{
$\begin{array}[b]{c|ccc}
\frac 1 4& \frac{101}{240} & -\frac{13}{60} & \frac{11}{240}  \\
\frac 1 2 & \frac{7}{12} & -\frac{1}{6}  & \frac{1}{12}     \\
\frac 3 4 & \frac{149}{240}  &   -\frac{7}{60}  & \frac{59}{240}          \\
\hline 
&\frac{2}{3}    &      -\frac{1}{3} &       \frac{2}{3} \\
\end{array}
$} & 4 & 4 2 1 \\
\midrule    
nIRK3oc & 3 & \parbox[c]{0.65\hsize}{
$\begin{array}[b]{c|ccc}
\frac 1 4& \frac{3}{16} & -\frac{1}{24} & \frac{5}{48}  \\
\frac 1 2 & 0 & -\frac{1}{6}  & \frac{2}{3}     \\
\frac 3 4 & \frac{9}{16}  &   -\frac{7}{24}  & \frac{23}{48}          \\
\hline 
&\frac{2}{3}    &      -\frac{1}{3} &       \frac{2}{3} \\
\end{array}
$} & 4 & 4 1 3 \\
\midrule
nIRK4o & 4 & \parbox[c]{0.65\hsize}{
$\begin{array}[b]{c|cccc}
\frac 1 5& \frac{1340}{3131} &       -\frac{538}{1343} & \frac{73}{336} & -\frac{6677}{149565}  \\
\frac 2 5&  \frac{1061}{1927} &     -\frac{823}{2137} & \frac{107}{336} & -\frac{47}{560}       \\
\frac 3 5 & \frac{571}{1053} &       -\frac{31}{112} & \frac{239}{560} &   -\frac{31}{336}     \\
\frac 4 5& \frac{761}{1513}   &  -\frac{2442}{13907} & \frac{743}{1680} &    \frac{902}{29713} \\ 
\hline 
  &\frac{11}{24} & \frac{1}{24} & \frac{1}{24} & \frac{11}{24}  \\
\end{array}
$}& 4 & 4 3 0 \\
\midrule
nIRK4oc & 4 & \parbox[c]{0.65\hsize}{
$\begin{array}[b]{c|cccc}
\frac 1 5& \frac{719}{1680} & -\frac{673}{168} & \frac{929}{4276} & -\frac{5}{112}  \\
\frac 2 5&  \frac{185}{336} &     -\frac{647}{1680} & \frac{485}{1523} & -\frac{47}{560}       \\
\frac 3 5 & \frac{911}{1680} &       -\frac{31}{112} & \frac{376}{881} &   -\frac{31}{336}     \\
\frac 4 5& \frac{169}{336}   & - \frac{59}{336} & \frac{789}{1784} &    \frac{17}{560} \\ 
\hline 
  &\frac{11}{24} & \frac{1}{24} & \frac{1}{24} & \frac{11}{24}  \\
\end{array}
$}
& 4 &4 0 3 \\
\bottomrule
\end{tabular}
\end{table}

\subsubsection{List of Gauss-Legendre, Radau and Lobatto IIIA methods}

Table~\ref{T-RL} contains Gauss-Legendre and Radau methods with $s=2,3$ stages, and  Lobatto IIIA  methods with $s=2,3,4,5$ stages. The first column contains the well-known name and the acronym of newly proposed approach which results in the same Butcher's tableau. See Subsection~\ref{SS-ngq}.
\begin{table}[!ht]
\caption{Gauss-Legendre, Radau and Lobatto IIIA methods}\label{T-RL}
\centering
\begin{tabular}[b]{cclcc}
\toprule
name  & $s$  &  Butcher's tableau & $p$  & $p$ $q$ $r$ \\
\midrule
 \parbox[c]{0.15\hsize}{\centering Gauss-Legendre \\
nIRK-G2 }  & 2 & \parbox[c]{0.65\hsize}{
$\begin{array}[b]{c|cc}
\frac{3-\sqrt{3}}{6} & \frac{1}{4}& \frac{1}{4}-\frac{\sqrt{3}}{6}   \\
\frac{3+\sqrt{3}}{6} &    \frac{1}{4}+\frac{\sqrt{3}}{6}   &\frac{1}{4}   \\ 
 \hline   
 & \frac{1}{2} & \frac{1}{2} \\
\end{array}
$}
& 4 & 4 2 2\\
\midrule
 \parbox[c]{0.15\hsize}{\centering Gauss-Legendre \\
nIRK-G3 }  & 3 & \parbox[c]{0.65\hsize}{
$\begin{array}[b]{c|ccc}
\frac{5-\sqrt{15}}{10} & \frac{5}{36} & \frac{2}{9}-\frac{\sqrt{15}}{15}   &\frac{5}{36}-\frac{\sqrt{15}}{30}    \\
\frac 1 2&   \frac{5}{36}+\frac{\sqrt{15}}{24}   &\frac{2}{9}&  \frac{5}{36}-\frac{\sqrt{15}}{24}    \\   
\frac{5+\sqrt{15}}{10}&  \frac{5}{36}+\frac{\sqrt{15}}{30}    &\frac{2}{9}+\frac{\sqrt{15}}{15}   & \frac{5}{36} \\
 \hline   
 & \frac{5}{18} & \frac{4}{9} & \frac{15}{18}\\
\end{array}
$}
& 6 & 6 3 3 \\
\midrule
 \parbox[c]{0.15\hsize}{\centering Radau I\\
nIRK-RI2 }  & 2 & \parbox[c]{0.65\hsize}{
$\begin{array}[b]{c|cc}
0& 0   &           0      \\
\frac{2}{3} &   \frac{1}{3} & \frac{1}{3} \\
\hline
& \frac 1 4 & \frac 3 4\\
\end{array}
$}
& 3 & 3 2 1 \\
\midrule
 \parbox[c]{0.15\hsize}{\centering Radau I\\
nIRK-RI3 }  & 3 & \parbox[c]{0.65\hsize}{
$\begin{array}[b]{c|ccc}
0& 0   &           0    & 0   \\
\frac{6-\sqrt{6}}{10} &   \frac{9+\sqrt{6}}{75} & \frac{24+\sqrt{6}}{120} &   \frac{168-73\sqrt{6}}{600}\\   
 \frac{6+\sqrt{6}}{10}&  \frac{9-\sqrt{6}}{75} & \frac{168+73\sqrt{6}}{600} & \frac{24-\sqrt{6}}{120}  \\
 \hline   
& \frac 1 9 & \frac{16+\sqrt{6}}{36} & \frac{16-\sqrt{6}}{36} \\
\end{array}
$}
& 5 & 5 3 2 \\
\midrule
 \parbox[c]{0.15\hsize}{\centering Radau IIA\\
nIRK-RII2 }  & 2 & \parbox[c]{0.65\hsize}{
$\begin{array}[b]{c|cc}
\frac{1}{3} &   \frac{5}{12} & -\frac{1}{12} \\
1 & \frac 3 4 & \frac 1 4     \\
\hline
& \frac 3 4 & \frac 1 4 \\
\end{array}
$}
& 3 & 3 2 1 \\
\midrule
 \parbox[c]{0.15\hsize}{\centering Radau IIA\\
nIRK-RII3 }  & 3 & \parbox[c]{0.65\hsize}{
$\begin{array}[b]{c|ccc}
\frac{4-\sqrt{6}}{10} & \frac{88-7\sqrt{6}}{360} & \frac{296-169\sqrt{6}}{1800}   &\frac{-2+3\sqrt{6}}{225}  \\
\frac{4+\sqrt{6}}{10}&   \frac{296+169\sqrt{6}}{1800} &\frac{88+7\sqrt{6}}{360}&   \frac{-2-3\sqrt{6}}{225} \\   
1&   \frac{16-\sqrt{6}}{36} & \frac{16+\sqrt{6}}{36} & \frac 1 9  \\
 \hline   
 & \frac{16-\sqrt{6}}{36} & \frac{16+\sqrt{6}}{36} & \frac 1 9\\
\end{array}
$}
& 5 & 5 3 2 \\
\midrule
\parbox[c]{0.15\hsize}{\centering 
Lobato IIIA\\
nIRK2\\
nIRK2c \\
sIRK2} & 2 &  \parbox[c]{0.65\hsize}{
$\begin{array}[b]{c|cc}
0& 0   &           0       \\
 1&    \frac 1 2    & \frac 1 2    \\
 \hline   
& \frac 1 2    & \frac 1 2     \\
\end{array}
$}
& 2 & 2 2 0 \\
\midrule
 \parbox[c]{0.15\hsize}{\centering Lobato IIIA\\
nIRK3 \\
nIRK3c\\
nIRK-L3 \\
sIRK3}  & 3 & \parbox[c]{0.65\hsize}{
$\begin{array}[b]{c|ccc}
0& 0   &           0    & 0   \\
\frac 1 2&   \frac{5}{24} &\frac 1 3 &   -\frac{1}{24}\\   
 1&   \frac 1 6 & \frac 2 3 & \frac 1 6  \\
 \hline   
&  \frac 1 6 & \frac 2 3 & \frac 1 6      \\
\end{array}
$}
& 4 & 4 3 1 \\
\midrule
 \parbox[c]{0.15\hsize}{\centering Lobato IIIA\\
nIRK-L4 \\}  & 4 & \parbox[c]{0.65\hsize}{
$\begin{array}[b]{c|cccc}
0& 0   &           0    & 0   &0 \\
\frac 1 2 - \frac{\sqrt{5}}{10}&   \frac{11+\sqrt{5}}{120}& \frac{25-\sqrt{5}}{120}&    \frac{25-13\sqrt{5}}{120} &  \frac{-1+\sqrt{5}}{120}\\   
\frac 1 2 + \frac{\sqrt{5}}{10} &   \frac{11-\sqrt{5}}{120}& \frac{25+13\sqrt{5}}{120}&    \frac{25+\sqrt{5}}{120} &  \frac{-1-\sqrt{5}}{120}\\ 
 1&   \frac{1}{12} & \frac{5}{12} & \frac{5}{12} & \frac{1}{12}\\
 \hline   
&  \frac{1}{12} & \frac{5}{12} & \frac{5}{12} & \frac{1}{12} \\
\end{array}
$}
& 6 & 6 4 2 \\
\midrule
 \parbox[c]{0.15\hsize}{\centering Lobato IIIA\\
nIRK-L5 \\}  & 5 & \parbox[c]{0.65\hsize}{
$\begin{array}[b]{c|ccccc}
0& 0   &           0    & 0   &0  &0\\
\frac 1 2 - \frac{\sqrt{21}}{14}&   \frac{119+3\sqrt{21}}{1960}& \frac{343-9\sqrt{21}}{2520}&    \frac{392-96\sqrt{21}}{2205} &  \frac{343-69\sqrt{21}}{2520} & \frac{-21+3\sqrt{21}}{1960}\\   
\frac 1 2 &   \frac{13}{320}& \frac{392+105\sqrt{21}}{2880}&    \frac{8}{45} &  \frac{392-105\sqrt{21}}{2880} & \frac{3}{320}\\ 
\frac 1 2 + \frac{\sqrt{21}}{14} & \frac{119-3\sqrt{21}}{1960}& \frac{343+69\sqrt{21}}{2520}&    \frac{392+96\sqrt{21}}{2205} &  \frac{343+9\sqrt{21}}{2520} & \frac{-21-3\sqrt{21}}{1960}\\ 
 1&   \frac{1}{20} &  \frac{49}{180} & \frac{16}{45} &\frac{49}{180} &\frac{1}{20}\\
 \hline   
&   \frac{1}{20} & \frac{49}{180} & \frac{16}{45} &\frac{49}{180} &\frac{1}{20} \\
\end{array}
$}
& 8 & 8 5 3 \\
\bottomrule
\end{tabular}
\end{table}

\subsubsection{List of standard collocation IRK}

Table~\ref{T-sIRK}, for completeness, contains the list of standard collocation IRK which were used in our numerical experiments. For the derivation of these methods see Subsubsection~\ref{SSS-sc}. 
\begin{table}[!ht]
\caption{standard collocation methods sIRK$s$o and sIRK$s$}\label{T-sIRK}
\centering
\begin{tabular}[b]{cclcc}
\toprule
name  & $s$  &  Butcher's tableau & $p$  & $p$ $q$ $r$ \\
\midrule
sIRK3o  & 3& \parbox[c]{0.65\hsize}{
$\begin{array}[b]{c|ccc}
\frac{1}{4} &   \frac{23}{48} & -\frac{1}{3} &   \frac{5}{48}\\   
 \frac{1}{2}&  \frac{7}{12} & -\frac{1}{6} & \frac{1}{12}\\  
\frac 3 4 & \frac{9}{16} & 0 & \frac{3}{16}\\
\hline
& \frac  2 3 & -\frac 1 3 & \frac  2 3 \\
\end{array}
$}
& 4 & 4 3 1 \\
\midrule
sIRK4o  & 4& \parbox[c]{0.65\hsize}{
$\begin{array}[b]{c|cccc}
\frac{1}{5} &   \frac{11}{24} & -\frac{59}{120} &   \frac{37}{120} & -\frac{3}{40}\\   
 \frac{2}{5}&  \frac{8}{15} & -\frac{1}{3} & \frac{4}{15} & -\frac{1}{15}\\  
\frac 3 5 & \frac{21}{40} & -\frac{9}{40} & \frac{3}{8} &-\frac{30}{40}\\
\frac 4 5 & \frac{8}{15} & -\frac{4}{15}& \frac{8}{15} & 0\\
\hline
& \frac{11}{24} & \frac{1}{24} & \frac{1}{24} & \frac{11}{24}\\
\end{array}
$}
& 4 & 4 4 0 \\
\midrule
sIRK4  & 4& \parbox[c]{0.65\hsize}{
$\begin{array}[b]{c|cccc}
0 &   0 & 0 & 0 & 0 \\   
 \frac{1}{3}&  \frac{1}{8} & \frac{19}{72} & -\frac{5}{72} & \frac{1}{72}\\  
\frac 2 3 & \frac{1}{9} & \frac{4}{9} & \frac{1}{9} &0\\
1 & \frac{1}{8} & \frac{3}{8}& \frac{3}{8} & \frac{1}{8}\\
\hline
& \frac{11}{8} & \frac{3}{8} & \frac{3}{8} & \frac{1}{8}\\
\end{array}
$}
& 4 & 4 4 0 \\
\midrule
sIRK5  & 5& \parbox[c]{0.65\hsize}{
$\begin{array}[b]{c|ccccc}
0 &   0 & 0 & 0 & 0  & 0\\ 
\frac{1}{4} &   \frac{251}{2880} & \frac{323}{1440} &   -\frac{11}{120} & \frac{53}{1440} & -\frac{19}{2880}\\   
 \frac{1}{2}&  \frac{29}{360} & \frac{31}{90} & \frac{1}{15} & \frac{1}{90}& -\frac{1}{360}\\  
\frac 3 4 & \frac{27}{320} & \frac{51}{160} & \frac{9}{40} &\frac{21}{160} & -\frac{3}{320}\\
1 & \frac{7}{90} & \frac{16}{45} & \frac{2}{15} & \frac{16}{45} & \frac{7}{90}\\
\hline
& \frac{7}{90} & \frac{16}{45} & \frac{2}{15} & \frac{16}{45} & \frac{7}{90}\\
\end{array}
$}
& 6 & 6 5 1 \\
\bottomrule
\end{tabular}
\end{table}

\end{document}